\documentclass{amsart}
\usepackage{amsmath}
\usepackage{graphics,graphicx,psfrag,color,type1cm}

\newcommand{\vip}{\vskip0.15cm}

\newcommand{\indiq}{{\rm 1 \hskip-4pt 1}}

\newcommand{\E}{{\mathbb{E}}}
\newcommand{\nn}{\mathbb{N}}
\newcommand{\zz}{\mathbb{Z}}
\newcommand{\tT}{\mathbb{T}}
\newcommand{\rr}{{\mathbb{R}}}
\newcommand{\cc}{{\mathbb{C}}}
\newcommand{\cE}{{\mathcal{E}}}
\newcommand{\cT}{{\mathcal{T}}}
\newcommand{\cH}{{\mathcal{H}}}

\newcommand{\cP}{{\mathcal{P}}}
\newcommand{\cN}{{\mathcal{N}}}
\newcommand{\cJ}{{\mathcal{J}}}
\newcommand{\cA}{{\mathcal{A}}}
\newcommand{\cS}{{\mathcal{S}}}
\newcommand{\cB}{{\mathcal{B}}}
\newcommand{\bQ}{{\bf {Q}}}
\newcommand{\bR}{{\bf {R}}}
\newcommand{\cI}{{\mathcal{I}}}
\newcommand{\hG}{{\widehat G}}

\newcommand{\expo}{{\mathcal{E}xp}}

\newcommand{\tcl}{\left[\!\left<}
\newcommand{\tcr}{\right>\!\right]}

\newcommand{\bex}{{\mathcal B}_{\rm ex}}
\newcommand{\psiex}{{\psi}_{\rm ex}}
\newcommand{\Ai}{{\rm Ai}}
\newcommand{\be}{{\bf e}}
\newcommand{\bn}{{\bf n}}
\newcommand{\bnel}{{{\bf n}_\infty}}

\newtheorem{theo}{\indent Theorem}
\newtheorem{prop}[theo]{\indent Proposition}
\newtheorem{rem}[theo]{\indent Remark}
\newtheorem{lem}[theo]{\indent Lemma}
\newtheorem{defin}[theo]{\indent Definition}

\newenvironment{preuve}{\vip \noindent {\it Proof}}{\hfill$\square$ \vip}

\begin{document}

\title[A mean field forest-fire model]
{Mean-field forest-fire models and pruning of random trees}
\author{Xavier Bressaud}
\author{Nicolas Fournier}
\thanks{Acknowledgments: The second author was supported during this work by
the grant from the Agence Nationale de la Recherche with reference 
ANR-08-BLAN-0220-01.}
\address{Xavier Bressaud: Universit\'e Paul Sabatier, Institut de 
Math\'ematiques de Toulouse, F-31062 Toulouse Cedex, France}
\email{bressaud@math.univ-toulouse.fr}
\address{Nicolas Fournier:
Laboratoire d'Analyse et de Math\'ematiques Appliqu\'ees, CNRS UMR 8050,
Universit\'e Paris-Est, 
61 avenue du G\'en\'eral de Gaulle, 94010 Cr\'eteil Cedex, France}
\email{nicolas.fournier@univ-paris12.fr}

\begin{abstract}
We consider a family of discrete coagulation-fragmentation equations 
closely related to the one-dimensional forest-fire model of statistical 
mechanics: each pair of particles with masses $i,j \in \nn$
merge together at rate $2$ to produce a single particle with mass $i+j$,
and each particle with mass $i$ breaks into $i$ particles with
mass $1$ at rate $(i-1)/n$. The (large) parameter $n$ controls the rate of 
ignition and there is also an acceleration factor
(depending on the total number of particles) in front
of the coagulation term.
We prove that for each $n\in \nn$, such a model has a unique equilibrium 
state and study in details the asymptotics of this equilibrium 
as $n\to \infty$: (I) the distribution of the mass of a typical
particle goes to the law of the
number of leaves of a critical binary Galton-Watson tree,
(II) the distribution of the mass of a typical size-biased particle 
converges, after rescaling,
to a limit profile, which we write explicitly in terms of the
zeroes of the Airy function and its derivative.
We also indicate how to simulate perfectly a typical particle and a size-biased
typical particle, which allows us to give some probabilistic interpretations 
of the above results in terms 
of pruned Galton-Watson trees and pruned continuum random trees.
\end{abstract}



\maketitle

\tableofcontents

\section{Introduction}

\subsection{The forest-fire model}
The forest-fire model of statistical mechanics has been introduced
by Henley \cite{h} and Drossel-Schwabl \cite{ds} in the context
of self-organized criticality. From the rigorous point of view, 
the one-dimensional forest-fire model has been studied by
Van den Berg-Jarai \cite{vdbj}, Brouwer-Pennanen \cite{bp}
and by the authors \cite{bf,bfnew}. We refer to the introduction of \cite{bfnew}
for many details. 

\vip

Let us now describe the one-dimensional forest-fire model:
on each site of $\zz$, seeds fall at rate $1$ and matches fall at rate $1/n$,
for some (large) $n\in \nn$. Each time a seed falls on a vacant site, this site
immediately becomes occupied (by a tree). Each time a match falls on an
occupied site, it immediately burns the corresponding occupied
connected component. 

\vip

From the point of view of self-organized criticality, 
it is interesting to study what happens when $n$ increases to infinity. Then
matches are very rare, but tree clusters are huge before they burn.
In \cite{bf}, we have established that after normalization,
the forest-fire process converges, as $n\to \infty$, to a scaling limit.

\subsection{A related mean-field model}
We now introduce a mean-field model formally related to the forest-fire
process, see \cite[Section 6]{bfold} for a similar study when $n=1$. 
Assume that each edge of $\zz$ has mass $1$. Say that
two adjacent edges $(i-1,i)$ and $(i,i+1)$ are glued if the site $i$
is occupied. Then adjacent clusters coalesce at rate $1$ and each
cluster with mass $k$ (containing $k$ edges and $k-1$ sites) breaks
up into $k$ clusters with mass $1$ at rate $(k-1)/n$. 
Denote by $c_k^n(t)$ the concentration (number per unit of length) of clusters
with mass $k\geq 1$ at time $t\geq 0$. Then the total mass should satisfy
$\sum_{k\geq 1} k c^n_k(t)=1$ for all $t\geq 0$. Neglecting
correlations (which is far from being justified), the family
$(c^n_k(t))_{t\geq 0, k\geq 1}$ would
satisfy the following system of differential
equations called $(CF_n)$:
\begin{align}
\frac{d}{dt}c^n_1(t)=& -2 c^n_1(t) + \frac{1}{n}\sum_{k \geq 2} k(k-1)c^n_k(t),
\label{eqt1}\\
\frac{d}{dt}c^n_k(t)=& -(2+(k-1)/n)c^n_k(t) + \frac{1}{\sum_{l\geq 1} c^n_l(t)}
\sum_{i=1}^{k-1} c^n_i(t)c^n_{k-i}(t) \quad (k\geq 2).\label{eqt2}
\end{align}
The first term on the RHS of (\ref{eqt1}) expresses a cluster
with mass $1$ disappears at rate $2$ (because it glues with each of 
its two neighbors at rate $1$); the second term on the RHS of (\ref{eqt1})
says that $k$ clusters with mass $1$ appear each time
a cluster with mass $k$ takes fire, which occurs at rate $(k-1)/n$.
The first term on the RHS of (\ref{eqt2}) explains that a cluster
with mass $k$ disappears at rate $2+(k-1)/n$ (because it glues with each of
its two neighbors at rate $1$ and takes fire at rate $(k-1)/n$).
Finally, the second term  
on the RHS of (\ref{eqt2}) says that when a seed falls between two clusters
with masses $i$ and $k-i$, a cluster with mass $k$ appears.
The number per unit of length of
pairs of neighbor clusters with masses $i$ and $k-i$ is nothing
but $c^n_i(t)c^n_{k-i}(t)/ \sum_{l\geq 1} c^n_l(t)$. Here we implicitly
use an independence
argument which is not valid for the true forest-fire model.

\vip

The system (\ref{eqt1})-(\ref{eqt2}) can almost be seen as a special 
coagulation-fragmentation equation, see e.g. Aizenman-Bak \cite{ab} and Carr
\cite{c}, where particles with masses $i$ and $j$ coalesce at constant rate
$K(i,j)=2$ and where particles with mass $i\geq 2$ break up into
$i$ particles with mass $1$ at rate $F(i;1,\dots ,1)=(i-1)/n$.
However there is the acceleration factor 
$1/ \sum_{l\geq 1} c^n_l(t)$ in front of the coagulation term.

\subsection{On the link between the two models}
The link between $(CF_n)$ and the forest-fire model is only formal.
Observe however that if there are only seeds {\bf or} only matches,
then the link is rigorous. 

\vip

(i) Assume that all sites of $\zz$ are initially vacant and that seeds fall
on each site of $\zz$ at rate $1$, so that each site is occupied at time $t$
with probability $1-e^{-t}$. Call $p_k(t)$ the probability that the edge $(0,1)$
belongs to a cluster with mass $k$ at time $t$. A simple computation
shows that $p_k(t)=k(1-e^{-t})^{k-1}e^{-2t}$. 
By space stationarity, the concentration $c_k(t)$ of
particles with mass $k$ at time $t$ satisfies 
$c_k(t)=p_k(t)/k=(1-e^{-t})^{k-1}e^{-2t}$. Then 
one easily checks that the family $(c_k(t))_{k\geq 1,t\geq 0}$ satisfies
(\ref{eqt1})-(\ref{eqt2}) with no fragmentation term (i.e. $n=\infty$).

\vip

(ii) The fragmentation term is linear and generates {\it a priori} 
no correlation. Assume that we have initially some given concentrations 
$(c_k(0))_{k\geq 1}$, so that the edge $(0,1)$ belongs to a cluster
with mass $k$ with probability $kc_k(0)$ and that only matches fall,
at rate $1/n$ on each site. Then the probability  
$p_k(t)$ that the edge $(0,1)$ belongs to a cluster with mass $k$ at time $t$
is simply given by $p_k(t)=kc_k(0)e^{-(k-1)t/n}$ if $k\geq 2$ (here
$e^{-(k-1)t/n}$ is the probability that no match has fallen on our cluster
before time $t$) and $p_1(t)=c_1(0)+\sum_{k\geq 2} kc_k(0)[1-e^{-(k-1)t/n}]$.
Writing $c_k(t)=p_k(t)/k$ as previously, we see that the family
$(c_k(t))_{k\geq 1,t\geq 0}$ satisfies the fragmentation equations
$\frac{d}{dt}c_1(t)=n^{-1}\sum_{k\geq 2}k(k-1)c_k(t)$ and, for $k \geq 2$,
$\frac{d}{dt}c_k(t)=- n^{-1}(k-1)c_k(t)$.

\vip

When one takes into account both coalescence and fragmentation,
the rigorous link between the two models breaks down: in the 
true forest-fire model, fragmentation (fires)
produces small clusters which are close to each other, so that a small
cluster has more chance to have small clusters as coalescence partners.
However, we have seen numerically in \cite[Section 6]{bfold}
that at equilibrium, in the special case where $n=1$, 
the two models are very close to each other.

\subsection{Motivation}
Initially, the motivation of the present study was to decide if 
(\ref{eqt1})-(\ref{eqt2}) is a good approximation of the true
forest-fire model, at least from a qualitative point of view: 
do we have the same scales and same features 
(as $n\to \infty$)? We will see that this is not really the case.
However, we believe that  (\ref{eqt1})-(\ref{eqt2}) is a very interesting model,
at least theoretically, since

$\bullet$ many explicit computations are possible for
(\ref{eqt1})-(\ref{eqt2}), 

$\bullet$ we observe self-organized criticality, 

$\bullet$ we show that two interesting points of view 
(size-biased and non size-biased particles' mass distribution)
lead to quite different conclusions,

$\bullet$ we have some probabilistic interpretations of our results.

\subsection{Summary of the main results of the paper}
We will show in this paper the existence of a unique 
equilibrium state $(c_k^n)_{k\geq 1}$ with total mass $\sum_{k\geq 1} kc^n_k=1$ 
for $(CF_n)$, for each $n\geq 1$ fixed and we study the asymptotics of rare fires $n\to \infty$.

\vip

(I) We show that the particles' mass distribution $(p^n_k)_{k\geq 1}$, defined by
$p^n_k=c_k^n/\sum_{l\geq 1}c^n_l$, goes
weakly to the law $(p_k)_{k\geq 1}$ of the number of
leaves of a critical binary Galton-Watson tree, which is explicit and satisfies
$p_k \sim (2 \sqrt \pi k^{3/2})^{-1}$ as $k\to \infty$;

\vip

(II) We prove that the size-biased particles' mass distribution $(kc_k^n)_{k\geq 1}$
goes weakly to, after normalization
of the masses by $n^{-2/3}$, to a continuous limit profile
$(xc(x))_{x\in (0,\infty)}$ with total mass $1$, for which we have two explicit expressions:
the first one involves the zeroes
of the Airy function and its derivative, while the Laplace transform 
of the Brownian excursion's area appears in the second one. 
Furthermore, we check that 
$c(x)\sim \kappa_1 x^{-3/2}$ as $x\to 0$ and $c(x) \sim \kappa_2 e^{-\kappa_3 x}$
as $x \to \infty$, for some positive explicit constants 
$\kappa_1,\kappa_2,\kappa_3$.

\vip

(III) We explain how to simulate
perfectly, for $n\geq 1$ fixed, a random variable $X_n$ with law
$(p_k^n)_{k\geq 1}$ and a random variable $Y_n$ with law $(kc_k^n)_{k\geq 1}$,
using some pruned Galton-Watson trees.

\vip

(IV) We deduce from (III) an easy probabilistic interpretation of point (I): roughly, $(p_k^n)_{k\geq 1}$
can be viewed as the law of the number of leaves of a pruned critical Galton-Watson tree, and we only
have to check that almost no pruning occurs when $n$ is very large.

\vip

(V) We also derive a probabilistic interpretation of (II), which seems quite natural, since the
Brownian excursion's area appears in the limit. This part is quite complicated. The main idea is that
that the pruned Galton-Watson tree of which the number of leaves is $(kc_k^n)_{k\geq 1}$-distributed
has a scaling limit, which is nothing but a pruned continuum random tree (precisely, a pruned
self-similar CRT, see Aldous \cite{a1}). We do not prove rigorously that this scaling limit arises.
However, we show that in some precise sense, the number of leaves of the pruned CRT under consideration
is indeed $xc(x)dx$-distributed, where $c$ is the limit profile found in point (II).

\vip

(VI) Finally, this pruned CRT  (heuristically) leads to a noticeable diffusion process
(rigorously) enjoying the strange property that its drift coefficient equals the Laplace
exponent of its inverse local time at $0$.

\vip

Let us discuss briefly points (I) and (II).
The particles' mass distribution is, roughly, the law of the mass of a particle
chosen uniformly at random. 
The size-biased particles' mass distribution is, roughly, the law of the mass 
of the particle containing a given {\it atom}, this atom being chosen
uniformly at random (think that a particle with mass $k$ is composed of
$k$ atoms). 

\vip

Point (I) says that if one picks a particle at random, then its mass $X_n$
is finite (uniformly in $n$) and goes in law, as $n\to \infty$,
to a {\it critical} probability distribution (with infinite expectation). 
Point (II) says that if one 
picks an atom at random, then the mass $Y_n$ of the particle including this 
atom is of order $n^{2/3}$ and $n^{-2/3}Y_n$ goes in law to an explicit probability
distribution with moments of all orders.

\subsection{Comments} Let us now comment on these results.

\vip

(a) Self-organized criticality, see Bak-Tang-Wiesenfield \cite{btw}
and Henley \cite{h}, is a popular concept in physics.
The main idea is the following: in statistical mechanics, there are
often some {\it critical} parameters, for which special features occur.
Consider e.g. the case of percolation in $\zz^2$, see Grimmett \cite{g}: 
for $p\in (0,1)$ fixed, open each edge independently with probability
$p$. There is a critical parameter $p_c\in (0,1)$ 
such that for 
$p\leq p_c$, there is a.s. no infinite open path, while for
$p> p_c$, there is a.s. one infinite open path. If $p<p_c$, all the 
open paths are small (the cluster-size distribution has an exponential decay). 
If $p>p_c$, there is only one infinite open path, which is huge, and all
the finite open paths are small. But if $p=p_c$, there are some large
finite open paths (with a heavy tail distribution).
And it seems that such phenomena, reminiscent from criticality, 
sometimes occur in nature, where nobody is here to finely tune the parameters.
Hence one looks for models in which criticality occurs naturally.
We refer to the introduction of \cite{bfnew} for many details.

\vip

(b) Thus we observe here self-organized criticality for the particles' mass 
distribution, since the limit distribution $(p_k)_{k\geq 1}$ has a heavy tail
and is indeed related to a {\it critical} binary Galton-Watson process.
Observe that point (I) is quite strange at first glance. Indeed, when
$n\to \infty$, the time-dependent equations (\ref{eqt1})-(\ref{eqt2}) 
tend to some
coagulation equations without fragmentation, for which there is no equilibrium
(because all the particles' masses tend to infinity). However, the equilibrium
state for (\ref{eqt1})-(\ref{eqt2}) tends to some non-trivial equilibrium state
as $n\to \infty$.
This is quite surprising: the limit (as $n\to \infty$) of the equilibrium
is not the equilibrium of the limit.
Observe that this is indeed {\it self-organized}
criticality: the {\it critical} 
Galton-Watson tree appears automatically in the limit $n\to \infty$.
A possible heuristic argument is that in some sense, in the limit $n\to \infty$,
only infinite clusters are destroyed. We thus let clusters grow as
much as they want, but we do not let them become infinite. Thus the system
reaches {\it by itself} a critical state.

\vip

(c) For the size-biased particles' mass distribution, 
we observe no self-organized criticality, since the limit profile has an 
exponential decay. The two points of view (size-biased or not) 
seem interesting and our results show that they really enjoy different 
features.
Let us insist on the fact that the particles' mass 
distribution converges without rescaling, while the 
size-biased particles' mass distribution converges after rescaling.
This is due to the
fact that there are many small particles and very few very large particles,
so that when one picks a particle at random, we get a rather small particle,
while when one picks an atom at random, it belongs to a rather large particle.
Mathematically, since the limiting particles' mass distribution has no 
expectation, it is no more possible to write properly the corresponding
size-biased distribution: a normalization is necessary.

\vip

(d) Let us mention the paper
of R\'ath-T\'oth \cite{rt}, who consider a forest-fire model 
on the complete graph. In a suitable regime, they obtain
the same critical distribution $(p_k)_{k\geq 1}$ as we do
(see \cite[Formula (14)]{rt}). But in their case,
this is the limit of the {\it size-biased} particles' mass distribution.
Their model rather corresponds to the case of a multiplicative coagulation
kernel (clusters of masses $k$ and $l$ coalesce at rate $kl$). 
Hence we exhibit, in some sense, a link between the Smoluchowski equation 
with constant kernel and the Smoluchowski equation with multiplicative kernel.
See Remark \ref{multconst} for a precise statement.
In the same spirit, recall the link found by Deaconu-Tanr\'e \cite{dt}
between multiplicative and additive coalescence.

\vip

(e) For the true forest-fire process,
we proved in \cite{bf} the presence of {\it macroscopic}
clusters, with masses of order $n / \log n$ and of {\it microscopic} clusters,
with masses of order $n^z$, for all values of $z \in [0,1)$. These microscopic
clusters interact with macroscopic clusters in that they limit the impact of
fires. Hence, while the scales are really different here, we observe something
similar: in some sense, the particles' mass distribution describes 
{\it microscopic} clusters, while the size-biased 
the particles' mass distribution describes {\it macroscopic} clusters.

\vip

(f) The trend to equilibrium for 
coagulation-fragmentation equations has been much studied,
see e.g. Aizenman-Bak \cite{ab} and Carr \cite{c}, under a 
reversibility condition, often called {\it detailed balance condition}.
Such a reversibility assumption cannot hold here, 
because particles merge by pairs and
break into an arbitrary large number of smaller clusters.
Without reversibility, much less is known: 
a special case has been studied by Dubowski-Stewart \cite{dst}
and a general result has been obtained in \cite{fm}
under a smallness condition saying that fragmentation is much stronger than 
coalescence. None of the
above results may apply to the present model (at least for $n$ slightly large).
For the specific model under study 
we are able to prove the uniqueness of the equilibrium, but not
the trend to equilibrium.

\vip

(g) Finally, let us mention that our probabilistic interpretation of (II) uses extensively
Aldous's theory
on continuum random trees \cite{a1,a2,a3}, which has been developed 
by Duquesne-Le Gall \cite{dlg,dlg2}.
Some pruning procedures of such CRTs have been introduced by Abraham-Delmas-Voisin \cite{adv}.
Roughly, the prune the CRT by choosing cut-points uniformly in the tree 
(uniformly according to the Lebesgue measure
on the tree).
Our pruning procedure is very different: we choose leaves (according to the measure on leaves), these leaves
send some cut-points on the branches (joining them to the root). Then we prune according to these cut-points,
in a suitable order. The cut-points sent by leaves belonging to subtrees previously pruned are deactivated.

\subsection{Outline of the paper}
The next section is devoted to the precise statement of points (I) and (II).
Section \ref{pr} contains the proofs of points (I) and (II), which are purely analytic. 
Using discrete pruned random trees, we indicate how to simulate
perfectly a typical particle and a size-biased typical particle in Section \ref{proba}.
Finally, we study in Section \ref{proba2} a possible scaling limit of the 
discrete tree representing a typical size-biased particle.

\section{Precise statements of the results}

For a $[0,\infty)$-valued sequence $u=(u_k)_{k\geq 1}$ 
and for $i\geq 0$, we put
$$
m_i(u):=\sum_{k\geq 1} k^i u_k.
$$

\begin{defin} Let $n\in \nn$. A sequence $c^n=(c_k^n)_{k\geq 1}$ of 
nonnegative real numbers is said solve $(E_n)$ if it is an equilibrium state
for $(CF_n)$ with total mass $1$:
\begin{align}\label{eq0}
m_1(c^n)=&1, \\
\label{eq2}
\left(2+ \frac{k-1}{n}\right)
c_k^n=& \frac{1}{m_0(c^n)}\sum_{i=1}^{k-1} c_i^n c_{k-i}^n \quad (k\geq 2).
\end{align}
Observe that it then automatically holds that
\begin{align}\label{eq1}
2c_1^n = \frac 1 n \sum_{k \geq 2} k(k-1)c^n_k =\frac{ m_2(c^n) - 1}{n}.
\end{align}
\end{defin}

To check this last claim, multiply (\ref{eq2})
by $k$, sum for $k\geq 2$, use (\ref{eq0})
and that $\sum_{k\geq 2} k\sum_{i=1}^{k-1}c^n_ic^n_{k-i}=
\sum_{k\geq 2} \sum_{i=1}^{k-1}(i+k-i)c^n_ic^n_{k-i}= 2\sum_{k,l\geq 1} k c^n_k c^n_l
=2m_0(c^n)$. One finds $2(1-c^n_1)+[m_2(c^n)-1]/n=2$, from which (\ref{eq1}) readily
follows.

\vip

To state our main results, we need some background on the Airy function $\Ai$.
We refer to Janson \cite[p 94]{j} and the references therein.
Recall that for $x\in \rr$, 
$\Ai(x)=\pi^{-1}\int_0^\infty \cos(t^3/3+xt)dt$ is the unique solution,
up to normalization, to the differential equation $\Ai''(x)=x\Ai(x)$ 
that is bounded for $x\geq 0$. It extends to an entire function.
All the 
zeroes of the Airy function and its derivative lie on negative real axis.
Let us denote by
$a_1'<0$ the largest negative zero of $\Ai'$ and by
$\dots < a_3<a_2<a_1<0$ the ordered zeroes of $\Ai$.
We know that $|a_1|\simeq 2.338$ and $|a_1'|\simeq 1.019$,
see Finch \cite{f2}. We also know that $|a_j|\sim (3\pi j/2)^{2/3}$
as $j\to \infty$, see Janson \cite[p 94]{j}.

\begin{theo}\label{mr1}
(i) For each $n\in \nn$, $(E_n)$ has a unique solution  
$c^n=(c_k^n)_{k\geq 1}$.

(ii) As $n\to \infty$, there hold
$$
m_0(c^n)\sim \frac 1 {|a_1'|n^{1/3}},
\quad m_2(c^n)\sim \frac{n^{2/3}}{|a_1'|}.
$$
\end{theo}

In some sense, $m_0(c^n)$ stands for the total concentration 
and $m_2(c^n)$ stands for the mean mass of 
(size-biased) clusters. For any fixed
$l\geq 0$, we also have shown that $m_{l+1}(c^n)\sim M_{l} n^{2l/n}$ for
some positive constant $M_l$, see Lemma \ref{tlm} below.

\begin{theo}\label{mr2}
For each $n\geq 1$, consider the unique solution $(c^n_k)_{k\geq 1}$ to $(E_n)$
and the corresponding particles' mass probability distribution $(p^n_k)_{k\geq 1}$
defined by $p_k^n=c_k^n/m_0(c^n)$. There holds
$$
\lim_{n\to \infty} \sum_{k\geq 1} |p_k^n-p_k| = 0,
\quad where \quad p_k:= \frac 2 {4^k k}\begin{pmatrix} 2k-2 \\ k-1 
\end{pmatrix}.
$$
The sequence $(p_k)_{k\geq 1}$ is the unique nonnegative solution to
\begin{align}\label{limsimple}
\sum_{k\geq 1} p_k=1,\quad p_k=\frac 1 2 \sum_{i=1}^{k-1}p_i p_{k-i} \quad (k\geq 2).
\end{align}
There holds, as $k\to \infty$, 
\begin{align}\label{eqpk}
p_k \sim \frac 1 {2\sqrt \pi k^{3/2}}.
\end{align}
\end{theo}

Formally, divide (\ref{eq2}) by $m_0(c^n)$ and
make $n$ tend to infinity: one gets  $2p_k=\sum_{i=1}^{k-1}p_i p_{k-i}$
for all $k\geq 2$. What is much more difficult (and quite surprising) is to establish
that no mass is lost at the limit.
Observe that (\ref{limsimple}) rewrites $\sum_{i\geq 1}p_i p_k
=(1/2)\sum_{i=1}^{k-1}p_i p_{k-i}$ (for all $k\geq 2$), which
corresponds to an equilibrium for a coagulation equation with constant kernel.
This is quite strange, since coagulation is a monotonic process, for which
no equilibrium should exist. The point is that in some sense, {\it infinite}
particles are broken into particles with mass $1$, in such a way that
$\sum_{k\geq 1} p_k=1$. 
Finally, we mention that $(p_k)_{k\geq 1}$ is the law of the number of leaves
of a critical binary Galton-Watson tree, which will be interpreted in Section \ref{proba}.

\begin{theo}\label{mr3}
For each $n\geq 1$, consider the unique solution $(c^n_k)_{k\geq 1}$ to $(E_n)$
and the corresponding size-biased particles' mass probability distribution 
$(kc^n_k)_{k\geq 1}$.
For any $\phi \in C([0,\infty))$ with at most polynomial growth, there holds
\begin{align*}
\lim_{n\to \infty} \sum_{k\geq 1} \phi(n^{-2/3}k)kc_k^n = \int_0^\infty \phi(x)xc(x)dx,
\end{align*}
where the profile $c:(0,\infty)\mapsto (0,\infty)$ is defined, for $x>0$, by
\begin{align*}
c(x)= |a_1'|^{-1} \exp\left(|a_1'|x\right) \sum_{j=1}^\infty 
\exp\left(-|a_j|x\right).
\end{align*}
The profile $c$ is of class $C^\infty$ on $(0,\infty)$,
has total mass $\int_0^\infty xc(x)dx=1$ and 
\begin{align*}
c(x) \stackrel {x \to 0} \sim \frac{1}{2\sqrt{\pi} |a_1'| x^{3/2}}
\quad and \quad c(x) \stackrel {x \to \infty} \sim |a_1'|^{-1} 
\exp\left( (|a_1'|-|a_1|)x \right).
\end{align*}
For any
$\phi \in C^1([0,\infty))$ such that $\phi$ and $\phi'$ have at most
polynomial growth,
\begin{align}\label{eqlim}
2|a_1'|\int_0^\infty \int_0^\infty x[\phi(x+y)-\phi(x)]c(x)c(y)dydx
= \int_0^\infty x^2[\phi(x)-\phi(0)] c(x)dx.
\end{align}
Denote by $\bex$ is the integral of the normalized Brownian excursion, 
see Revuz-Yor \cite[Chapter XII]{ry}. For all $x>0$, 
\begin{align*}
c(x) =\frac{\exp\left(|a_1'|x\right)}
{2\sqrt{\pi} |a_1'| x^{3/2}}\E\left[e^{-\sqrt{2} x^{3/2}\bex } \right].
\end{align*}
Finally, for all $q \in (a_1-a_1',\infty)$ (recall that $a_1-a_1'<0$),
\begin{align}\label{lala}
\ell(q):=\int_0^\infty (1-e^{-qx})c(x)dx = \frac{-\Ai'(q+a_1')}{|a_1'|\Ai(q+a_1')}.
\end{align}
\end{theo}

Since the mean mass of a typical (size-biased) particle is of order
$n^{2/3}$ by Theorem \ref{mr1}-(ii), it is natural to
rescale the particles' masses by a factor $n^{-2/3}$. Here we state
that under this scale, there is indeed a limit profile
and we give some information about this profile.
The link with the Brownian excursion's area will interpreted
in Section \ref{proba2}.

\section{Analytic proofs}\label{pr}

For each $n \in \nn$, we introduce the sequence
$(\alpha_k^n)_{k\geq 1}$, defined recursively by
\begin{align}\label{dfalpha}
\alpha_1^n=1,\quad \left(2+(k-1)/n\right)\alpha_{k}^n=\sum_{i=1}^{k-1} 
\alpha_i^n \alpha_{k-i}^n \quad (k\geq 2).
\end{align}
We also introduce its generating function $f_n$, defined for $q\geq 0$ by
\begin{align}\label{generaf}
f_n(q)=\sum_{k\geq 1} \alpha^n_k q^k.
\end{align}
Obviously, $f_n$ is increasing on $[0,\infty)$ and takes its values in
$[0,\infty)\cup \{\infty\}$.
The unique solution to $(E_n)$ can be expressed in terms
of this sequence.

\begin{lem}\label{equi}
Fix $n\in\nn$.
Assume that there exists a (necessarily unique) 
$q_n>0$ such that $f_n(q_n)=1$. Assume furthermore that $f_n'(q_n)<\infty$.
Then there is a unique solution to $(E_n)$
and it is given by
\begin{align}\label{ettac}
c^n_k = \alpha^n_k  q_n^{k-1}/f_n'(q_n) \quad (k\geq 1).
\end{align}
Furthermore, there holds $m_0(c^n)= 1/(q_n f_n'(q_n))$.
\end{lem}

\begin{proof} We break the proof into 3 steps.
\vip

{\it Step 1.} A simple computation shows that for any fixed $r>0$, $x>0$,
the sequence defined recursively by
\begin{align*}
y_1=x, \quad \left(2+(k-1)/n\right) y_{k}=\frac{1}{r}\sum_{i=1}^{k-1} 
y_i y_{k-i} \quad (k\geq 2)
\end{align*}
is given by $y_k=\alpha^n_k x (x/r)^{k-1}$.
\vip

{\it Step 2.} Consider a solution $(c^n_k)_{k\geq 1}$ to $(E_n)$. 
Then 
due to (\ref{eq2}) and Step 1 (write $r=m_0(c^n)$ and $x=c^n_1$), 
$c^n_k=\alpha^n_k c_1^n (c_1^n/m_0(c^n))^{k-1}$.
We deduce that 
$$
1 = \frac 1 {m_0(c^n)}\sum_{k\geq 1} c^n_k = \sum_{k\geq 1}
\alpha^n_k (c_1^n/m_0(c^n))^{k}= f_n( c_1^n/m_0(c^n)).
$$
Consequently, $c_1^n/m_0(c^n)=q_n$, whence
$c^n_k=\alpha^n_k c_1^n q_n^{k-1}$. Next we know that $m_1(c^n)=1$, so that
$f_n'(q_n)=\sum_{k\geq 1} k \alpha^n_k q_n^{k-1}
=m_1(c^n)/c_1^n=1/c_1^n$. Consequently, $c_1^n =1/f_n'(q_n)$ and
thus $c^n_k=\alpha^n_k q_n^{k-1}/f_n'(q_n)$  as desired.

\vip

{\it Step 3.} Let us finally check that $c^n$ as defined by (\ref{ettac})
is indeed solution to $(E_n)$ and that 
it satisfies $m_0(c^n)=1/(q_nf_n'(q_n))$.
First, 
$$
m_0(c^n)=\frac{1}{f_n'(q_n)}\sum_{k\geq 1} \alpha^n_k q_n^{k-1}=
\frac{f_n(q_n)}{q_nf_n'(q_n)} = \frac{1}{q_nf_n'(q_n)}.
$$
Next, (\ref{eq0}) holds, since
$$
m_1(c^n)=\frac{1}{f_n'(q_n)}\sum_{k\geq 1} k \alpha^n_k q_n^{k-1}
=\frac{f_n'(q_n)}{f_n'(q_n)}=1.
$$
Rewriting $c^n_k=\alpha^n_k x (x/r)^{k-1}$ with
$x=1/f_n'(q_n)$ and $r=1/(q_nf_n'(q_n))$, Step 1 implies that for $k\geq 2$,
\begin{align*}
\left(2+(k-1)/n\right)c^n_k = q_n f_n'(q_n) \sum_{i=1}^{k-1} c^n_i c^n_{k-i}
= \frac{1}{m_0(c^n)} \sum_{i=1}^{k-1} c^n_i c^n_{k-i},
\end{align*}
whence (\ref{eq2}).
\end{proof}

To go on, we need some background on Bessel functions of the first kind.
Recall that for $k \geq 0$ and $z \in \cc$,
\begin{align}\label{dfbessel}
J_k(z)= \frac{z^k}{2^k}\sum_{l\geq 0} \frac{(-1)^l z^{2l} }{4^l l ! (l+k)!}.
\end{align}
We have the following recurrence relations, see \cite[Section 4.6]{aar}: 
for $k\geq 0$, $x\in \rr$,
\begin{align}
\label{bf1} &J_{k}'(x)=\frac{k}{x}J_{k}(x) -J_{k+1}(x) \\
\label{bf2} &J_{k+2}(x)=\frac{2(k+1)}{x}J_{k+1}(x)-J_{k}(x),\\
\label{bf3} &\frac{d}{dx} \left(x^{k+1}J_{k+1}(x)\right)=x^{k+1}J_{k}(x).
\end{align}
It is known, see \cite[Section 4.14]{aar}, 
that all the zeroes of $J_k$ are real.
For all $k\geq 1$, we denote by $j_k$ the first positive zero of $J_k$
and by $j_k'$ the first zero of $J_k'$.
The sequence $(j_k)_{k \geq 0}$ is increasing,
see \cite[Section 4.14]{aar}. Furthermore,
$0<j_k'<j_k$ for all $k\geq 1$, see \cite[page 3]{f1}.
We will also use that
there exists a constant $C>0$ such that for all $k\geq 1$, 
see \cite[pages 2-3]{f1},
\begin{align}\label{bz}
k + 2^{-1/3}|a_1'| k^{1/3} < j_k' < k + 2^{-1/3}|a_1'| k^{1/3}+ C k^{-1/3},
\end{align}
where $a_1'$ is, as previously defined, the largest negative zero
of $\Ai'$.

\begin{lem}\label{exiq}
Let $n\in\nn$ be fixed. The radius of convergence of the entire series
$f_n$ is $r_n=(j_{2n-1}/(2n))^2/2$. For all $x\in [0,r_n)$, there holds
\begin{align}\label{fnexpl}
f_n(x)=\sqrt{2x} \frac{J_{2n}(2n\sqrt{2x})}{J_{2n-1}(2n\sqrt{2x})}.
\end{align}
There exists a unique $q_n \in (0,r_n)$ such that $f_n(q_n)=1$. 
There holds $f_n'(q_n)=(n/q_n)(2q_n-1+1/n) \in (0,\infty)$.
Finally, as $n\to \infty$,
\begin{align*}
q_n= \frac{1}{2}\left[1+|a_1'|n^{-2/3} \right] + O(n^{-1}).
\end{align*}
\end{lem}

\begin{proof} We fix $n\in \nn$ and divide the proof into five steps.

\vip

{\it Step 1.} Put $s_n=(j_{2n-1}/(2n))^2/2$.
By definition of $j_{2n-1}$ and since $J_{2n-1}$ is odd on $\rr$
and has no complex zeroes, $J_{2n-1}(2n\sqrt{2z})$
does not vanish on $\{0<|z|<s_n\} \subset \cc$. We thus may define
$g_n(z):=\sqrt{2z} J_{2n}(2n\sqrt{2z} )/J_{2n-1}(2n\sqrt{2z})$ on $\{0<|z|<s_n\}$.
Using (\ref{dfbessel}), one easily checks that, as $z\to 0$,
$g_n(z) \sim z$, so that finally, $g_n(z)$ is holomorphic on the disc
$\{|z|<s_n\}$. Write $g_n(z)=\sum_{k\geq 0} \beta^n_k z^k$. Since
$g_n(z) \sim z$ near $0$, we deduce that $\beta^n_0=0$ and $\beta^n_1=1$.

\vip

{\it Step 2.} We now show that for all $x\in [0,s_n)$, there holds
\begin{align}\label{ed}
xg_n'(x)/n + (2-1/n) g_n(x)=g_n^2(x)+2x.
\end{align}
Write $g_n(x)=h_n(2n \sqrt{2x})/(2n)$, where $h_n(y)=yJ_{2n}(y)/J_{2n-1}(y)$.
Observing that 
$h_n(y)=(y^{2n}J_{2n}(y))/(y^{2n-1}J_{2n-1}(y))$ and using
(\ref{bf3}) and then (\ref{bf2}),
\begin{align*}
h_n'(y)=& \frac{y^{2n}J_{2n-1}(y)}{y^{2n-1}J_{2n-1}(y)} - 
\frac{y^{2n}J_{2n}(y) y^{2n-1}J_{2n-2}(y)}{(y^{2n-1}J_{2n-1}(y))^2}\\
=& y - h_n(y) \frac{J_{2n-2}(y)}{J_{2n-1}(y)}\\
=&y - h_n(y) \left[\frac{-J_{2n}(y)}{J_{2n-1}(y)} + \frac{2(2n-1)}{y} \right]\\
=& y + \frac{h_n^2(y)}{y} - 2 (2n-1)\frac{h_n(y)}{y}. 
\end{align*}
But $g_n'(x)=h_n'(2n\sqrt{2x})/\sqrt{2x}$, whence
\begin{align*}
\frac{xg_n'(x)} n
=&\frac{x}{n\sqrt{2x}}\left[2n\sqrt{2x} 
+\frac{h_n^2(2n\sqrt{2x})}{2n\sqrt{2x}} 
- 2 (2n-1)\frac{h_n(2n\sqrt{2x})}{2n\sqrt{2x}}  \right] \\
=& 2x+ g_n^2(x)-(2-1/n)g_n(x).
\end{align*}
\vip

{\it Step 3.} Let us check that the sequence $(\beta^n_k)_{k\geq 1}$
satisfies (\ref{dfalpha}). We already know that $\beta^n_1=1$.
Using (\ref{ed}),
\begin{align*}
\sum_{k\geq 1} k(\beta^n_k/n) x^k + \sum_{k\geq 1} (2-1/n) \beta^n_k x^k = 
\sum_{k\geq 2} x^k \sum_{i=1}^{k-1} \beta^n_i\beta^n_{k-i} + 2x  .
\end{align*}
Thus for all $k\geq 2$, $(k/n) \beta^n_k+  (2-1/n) \beta^n_k 
= \sum_{i=1}^{k-1} \beta^n_i\beta^n_{k-i}$ as desired.
Consequently,
$(\beta^n_k)_{k\geq 1}=(\alpha^n_k)_{k\geq 1}$, whence $f_n=g_n$ and 
$r_n=s_n= (j_{2n-1}/(2n))^2/2$.

\vip

{\it Step 4.} We know that $f_n$ is $C^\infty$ and increasing on $[0,r_n)$, that
$f_n(0)=0$ and that $\lim_{x \to r_n-} f_n(x)=\infty$ 
(due to (\ref{fnexpl}) and because $j_{2n}>j_{2n-1}$). Hence,
there exists a unique $q_n\in [0,r_n)$ such that $f_n(q_n)=1$ and
we have $0<f_n'(q_n)<\infty$. Applying (\ref{ed}) at $x=q_n$
(recall that $f_n=g_n$), we deduce that $q_nf_n'(q_n)/n+(2-1/n)=1+2q_n$,
so that $f_n'(q_n)=(n/q_n)[2q_n-1+1/n]$.

\vip

{\it Step 5.} Put $\gamma_n=2n\sqrt{2q_n}\in (0,2n\sqrt{2 r_n})=(0,j_{2n-1})$.
Then $f_n(q_n)=1$ rewrites $\gamma_n J_{2n}(\gamma_n)=2n J_{2n-1}(\gamma_n)$. 
We now prove that $j_{2n-1}'\leq \gamma_n\leq j_{2n}'$.

\vip

$\bullet$ First, using (\ref{bf1}) with $k=2n-1$, we get
$x J_{2n}(x)= (2n-1)J_{2n-1}(x) - xJ_{2n-1}'(x)$.
Since $\gamma_n J_{2n}(\gamma_n)=2n J_{2n-1}(\gamma_n)$, we find 
$J_{2n-1}(\gamma_n)+\gamma_nJ'_{2n-1}(\gamma_n)=0$. Thus $\gamma_n\geq j'_{2n-1}$,
because for $0<x<j_{2n-1}'<j_{2n-1}$, $J_{2n-1}(x)$ and 
$J'_{2n-1}(x)$ are positive. 

\vip

$\bullet$ We next show that $\gamma_n \leq j'_{2n}$. We already know
that $\gamma_n < j_{2n-1}$. We thus assume below that $j'_{2n} < j_{2n-1}$, because
else, there is nothing to do.
There holds $(2n/\gamma_n)[\gamma_n^{2n} J_{2n-1}(\gamma_n) ]
/[\gamma_n^{2n} J_{2n}(\gamma_n)]=1$, whence 
$(2n/\gamma_n)[\log (\gamma_n^{2n} J_{2n}(\gamma_n))]'=1$ by (\ref{bf3}).
Consequently, 
$$
[\log J_{2n}(\gamma_n)]'=(\gamma_n/2n)-(2n/\gamma_n).
$$
We know by (\ref{bz}) that $j'_{2n}>2n$. Hence $J_{2n}'>0$ and thus
$[\log J_{2n}]'>0$ on $[0,2n]$. Thus $\gamma_n>2n$, because 
for $x\leq 2n$, we have $[\log J_{2n}(x)]'>0$ and 
$(x/2n)-(2n/x)\leq 0$.
Hence $(\gamma_n/2n)-(2n/\gamma_n)>0$. Thus 
$\gamma_n < j_{2n}'$, because for $x\in (j'_{2n},j_{2n-1})$, $J_{2n}'(x)<0$
so that $[\log J_{2n}(x)]'<0$.

\vip

We thus have checked that $\gamma_n \in (j'_{2n-1},j'_{2n})$. Using (\ref{bz}),
we deduce that
\begin{align*}
(2n-1)+2^{-1/3}|a_1'|(2n-1)^{1/3} < \gamma_n < 2n+2^{-1/3}|a_1'|(2n)^{1/3} 
+ C(2n)^{-1/3},
\end{align*}
so that $\gamma_n = 2n + |a_1'|n^{1/3} + O(1)$. Recalling that
$q_n=(\gamma_n/(2n))^2/2$, we easily deduce that 
$q_n=(1+ |a_1'|n^{-2/3})/2 +O(n^{-1})$
as desired.
\end{proof}

We now have all the tools to give the

\begin{preuve} {\it of Theorem \ref{mr1}-(i).} 
Fix $n\in \nn$. Due to Lemma \ref{exiq}, there is a unique $q_n>0$ such
that $f_n(q_n)=1$ and $f_n'(q_n)<\infty$. Applying Lemma \ref{equi},
we deduce the existence and uniqueness of a solution to $(E_n)$.
\end{preuve}

Let us now give two weak forms of the equilibrium equation $(E_n)$.

\begin{lem}\label{wf}
For $n\in\nn$, consider the unique solution $(c^n_k)_{k\geq 1}$
to $(E_n)$. 
For 
any $\phi,\psi:\nn\mapsto \rr$ with at most polynomial growth,
\begin{align}\label{wf1}
&\frac{1}{m_0(c^n)} \sum_{k,l \geq 1} \left[\phi(k+l)-\phi(k)-\phi(l) \right]
c^n_k c^n_l
= \frac{1}{n} \sum_{k \geq 2} \left[\phi(k)-k\phi(1) \right](k-1)c^n_k, \\
\label{wf2}
&\frac{2}{m_0(c^n)} \sum_{k,l \geq 1} \left[\psi(k+l)-\psi(k)\right]
k c^n_k c^n_l
= \frac{1}{n} \sum_{k \geq 2} \left[\psi(k)-\psi(1) \right]k(k-1)c^n_k.
\end{align}
\end{lem}

\begin{proof}
Let $n\in\nn$ be fixed.
\vip

{\it Step 1.} We first check that there is $u_n>1$ such that 
$\sum_{k\geq 1}u_n^k c^n_k<\infty$. This allows us to justify the convergence
of all the series in Steps 2 and 3 below. 
We know (see Lemma \ref{equi}) that 
$c^n_k=\alpha^n_k q_n^{k-1}/f_n'(q_n)$. Hence for any $u>0$,
$\sum_{k\geq 1} u^k c^n_k = f_n(q_n u)/(q_nf_n'(q_n))$. Since the radius of
convergence $r_n$ of the entire series $f_n$ 
satisfies $r_n>q_n$ by Lemma \ref{exiq}, 
the result follows: choose $u_n>1$ such that $q_n u_n < r_n$.

\vip

{\it Step 2.} We now prove (\ref{wf1}). 
Multiply (\ref{eq1})-(\ref{eq2}) 
by $\phi(k)$ and sum for $k\geq 1$.
We get 
\begin{align*}
&\frac{1}{m_0(c^n)}\left[\sum_{k\geq 2} \phi(k) \sum_{i=1}^{k-1}c_i^nc_{k-i}^n
- 2m_0(c^n)\sum_{k\geq 1} \phi(k)c_k^n\right] \\
=& \frac{1}{n}\left[\sum_{k\geq 2} \phi(k)(k-1)c_k^n
- [m_2(c^n)-1]\phi(1)\right]. 
\end{align*}
But
$\sum_{k\geq 2} \phi(k) \sum_{i=1}^{k-1}c_i^nc_{k-i}^n=
\sum_{i,j \geq 1} \phi(i+j)c^n_i c^n_j$. Furthermore, there holds
$2m_0(c^n)\sum_{k\geq 1} \phi(k)c_k^n=\sum_{i,j \geq 1} \left(\phi(i)+\phi(j) \right)
c^n_i c^n_j$, as well as $[m_2(c^n)-1]\phi(1) = [m_2(c^n)-m_1(c^n)]\phi(1)=
\sum_{k\geq 1} k\phi(1)(k-1)c_k^n$. This ends the proof of (\ref{wf1}).

\vip

{\it Step 3.}
To check (\ref{wf2}), it suffices to apply (\ref{wf1}) to the function
$\phi(k)=k\psi(k)$ and to use that by symmetry,
$\sum_{k,l\geq 1} [(k+l)\psi(k+l)-k\psi(k)-l\psi(l)]c^n_kc^n_l=
2\sum_{k,l\geq 1} k [\psi(k+l)-\psi(k)]c^n_kc^n_l$.
\end{proof}

\begin{rem}\label{multconst}
For $n\in\nn$, consider the solution $(c^n_k)_{k\geq 1}$ to $(E_n)$,
write $p^n_k=c_k^n/m_0(c^n)$ and then $d^n_k=p^n_k/k$. Then one easily
checks, starting from (\ref{wf1}), that the sequence $(d^n_k)_{k\geq 1}$
solves, for all $\phi:\nn\mapsto \rr$ with at most polynomial growth,
\begin{align*}
\frac12\sum_{k,l\geq 1}\left[\phi(k+l)-\phi(k)-\phi(l)\right] k l d^n_kd^n_l
= \sum_{k\geq 1}  \left[\phi(k)-\phi(1) \right]\frac{k(k-1)}{2n} d^n_k
\end{align*}
and has total mass $\sum_{k\geq 1} k d^n_k =1$. Hence $(d^n_k)_{k\geq 1}$
is an equilibrium state for a coagulation-fragmentation equation
with multiplicative coagulation kernel (particles with masses $k,l$ merge at
rate $kl$) and where each particle with mass $k$ breaks into $k$
particles with mass $1$ at rate $k(k-1)/(2n)$. This explains the similarities
between Theorem \ref{mr2} and the results found by R\'ath-T\'oth \cite{rt}:
the (non size-biased) particles' mass distribution $(p^n_k)_{k\geq 1}$
has the same limit, as $n\to \infty$, as the size-biased
particles' mass distribution of the model considered by R\'ath-T\'oth,
see \cite[Eq. (15)]{rt}. Observe however that the fragmentation rate 
in \cite{rt} is rather $(k-1)/n$, which thus differs from $k(k-1)/(2n)$.
What seems important is just that roughly, for $n$ very large,
only huge particles break down into atoms.
\end{rem}

We are ready to handle the

\begin{preuve} {\it of Theorem \ref{mr1}-(ii).}
First, we know from Lemmas \ref{equi} and \ref{exiq} 
that 
\begin{align*}
m_0(c^n)=\frac 1{q_nf_n'(q_n)}=\frac 1 {n(2q_n-1+1/n)}
=\frac 1 {n[|a_1'|n^{-2/3} + O(1/n)]} \sim \frac1{|a_1'|n^{1/3}}.
\end{align*}
Next we use (\ref{wf1}) with $\phi(k)=-1$. This gives
\begin{equation}\label{ettac2}
m_0(c^n)=\frac{m_2(c^n)+m_0(c^n)-2m_1(c^n)} n,
\end{equation}
whence, since $m_1(c^n)=1$,
$$
m_2(c^n)= (n-1)m_0(c^n) +2 \sim \frac{n^{2/3}}{|a_1'|}.
$$
Theorem \ref{mr1}-(ii) is established.
\end{preuve}

We can now prove the convergence of the particles' mass distribution.

\begin{preuve} {\it of Theorem \ref{mr2}.}
Let us put, for $k\geq 1$, 
$$p_k= \frac 2 {4^k k}\begin{pmatrix} 2k-2 \\ k-1 \end{pmatrix}.  
$$
Using the Stirling formula, one immediately checks (\ref{eqpk}).
Next, recall that the Catalan numbers, defined by
$C_i = \frac{1}{i+1}\begin{pmatrix} 2i \\ i \end{pmatrix}$ 
for all $i\geq 0$, satisfy, see \cite{w}
\begin{align*}
&C_{i}=\sum_{j=0}^{i-1} C_j C_{i-1-j}, \quad (i\geq 1),\\
&\sum_{i\geq 0} C_i x^i = (1-\sqrt{1-4x})/(2x),\quad x\in [0,1/4].
\end{align*} 
Observing that $p_k=2. 4^{-k}C_{k-1}$, we easily deduce that
$p_k= \frac 1 2 \sum_{i=1}^{k-1} p_i p_{k-i}$ for $k\geq 2$, as well
as $\sum_{k\geq 1} p_k=2 \sum_{k\geq 1} C_{k-1}(1/4)^k= [\sum_{i\geq 0} C_{i}(1/4)^i]/2
=1$.
\vip We now check that (\ref{limsimple}) has at most one nonnegative solution. 
To this end, it suffices to show that (\ref{limsimple}) implies that
$p_1=1/2$ (because this will determine the value of $p_2=(1/2)p_1^2$,
of $p_3=(1/2)[p_1p_2+p_2p_1]$ and so on). To this end, it suffices to write
$p_1=1-\sum_{k\geq 2}p_k=1-(1/2)\sum_{k\geq 2}\sum_{i=1}^{k-1}p_ip_{k-i}=
1-(1/2)(\sum_{k\geq 1} p_k)^2=1-1/2=1/2$.

\vip
It only remains to prove that $\lim_{n} \sum_{k\geq 1}|p^n_k-p_k|=0$.
Since $\sum_{k\geq 1} p^n_k=\sum_{k\geq 1} p_k=1$, it classically suffices
to prove that $\lim_n p^n_k=p_k$ for all $k\geq 1$.
First of all, we observe from (\ref{eq1}) and Theorem \ref{mr1}-(ii) that
\begin{align*}
p^n_1 = \frac{c^n_1}{m_0(c^n)}=\frac{m_2(c^n)-1}{2n m_0(c^n)}\sim 
\frac{n^{2/3}/|a_1'|}{2n/(|a_1'|n^{1/3})} \to \frac 1 2=p_1
\end{align*}
as $n\to \infty$. Next, we work by induction on $k$. Assume thus that for some
$k\geq 2$, $\lim_n p^n_l=p_l$ for $l=1,\dots,k-1$. Then, using (\ref{eq2}),
we deduce that
\begin{align*}
p^n_k = \frac{c^n_k}{m_0(c^n)}=\frac{1}{2+(k-1)/n}\sum_{i=1}^{k-1} p^n_i p^n_{k-i}
\to \frac{1}{2}\sum_{i=1}^{k-1} p_i p_{k-i} = p_k
\end{align*}
as $n\to \infty$. This concludes the proof.
\end{preuve}

We now study the size-biased particles' mass distribution. 
We start with the computation of all the 
moments and deduce a convergence result.

\begin{lem}\label{tlm}
For each $n\in\nn$, consider the unique solution $(c^n_k)_{k\geq 1}$
to $(E_n)$ and the probability measure
$\mu_n=\sum_{k\geq 1} k c^n_k \delta_{n^{-2/3}k}$ on $(0,\infty)$.

(i) For any $i\geq 1$,
$$
m_{i+1}(c^n) \stackrel {n\to \infty} \sim M_i n^{2i/3}, 
$$
where the sequence
$(M_i)_{i\geq 0}$ is defined by $M_0=1$, $M_1=1/|a_1'|$ and, for $i\geq 1$,
\begin{align*}
M_{i+1} = 2|a_1'| \sum_{j=0}^{i-1} \begin{pmatrix} i \\ j \end{pmatrix} 
M_j M_{i-j-1}.
\end{align*}

(ii) There is a probability measure $\mu$ on $[0,\infty)$ such that
for any $\phi\in C([0,\infty))$ with at most polynomial growth, 
\begin{align*}
\sum_{k\geq 1} \phi(n^{-2/3} k) kc_k^n = 
\int_0^\infty \phi(x)\mu_n(dx) \stackrel{n\to \infty} \to 
\int_0^\infty 
\phi(x)\mu(dx).
\end{align*}
This probability measure satisfies, for all $i\geq 0$, $\int_0^\infty
x^i \mu(dx)=M_i$.
\end{lem}

\begin{proof}
Applying (\ref{wf2}) with $\phi(k)=k^i$ with $i\geq 1$,
we easily get
\begin{align*}
\frac{2}{m_0(c^n)} \sum_{j=0}^{i-1} \begin{pmatrix}i\\j \end{pmatrix}
m_{j+1}(c^n)m_{i-j}(c^n)
=\frac 1 n \left(m_{i+2}(c^n)- m_{i+1}(c^n)-m_2(c^n)+1\right),
\end{align*}
so that
\begin{align*}
m_{i+2}(c^n)=m_{i+1}(c^n)+m_2(c^n)-1 + \frac{2n}{m_0(c^n)} 
\sum_{j=0}^{i-1} \begin{pmatrix}i\\j \end{pmatrix}
m_{j+1}(c^n)m_{i-j}(c^n).
\end{align*}
From this and the fact that we already know that $m_2(c^n)\sim M_1n^{2/3}$
and that $m_0(c^n)\sim n^{-1/3}/|a_1'|$, one can easily 
check point (i) by induction.

\vip

Next, it holds that $\int_0^\infty x^i \mu_n(dx)= n^{-2i/3}m_{i+1}(c^n)$
for all $i\geq 1$. Consequently, we know from point (i) that
$\lim_{n\to \infty} \int_0^\infty x^i \mu_n(dx) = M_i$ for any $i\geq 1$ .
In particular, $\sup_n \int_0^\infty x \mu_n(dx)<\infty$, so that
the sequence $(\mu_n)_{n\geq 1}$ is tight. Furthermore, any limit point $\mu$ 
is a probability measure on $[0,\infty)$
satisfying $\int_0^\infty x^i \mu(dx)=M_i$ for all $i\geq 1$. 

\vip

But such a probability measure $\mu$ is unique: an immediate (and rough)
induction using that $|a_1'|>1$ 
shows that $M_i\leq (2|a_1'|)^i i !$ for all $i\geq 1$. This 
implies the finiteness of an exponential moment for $\mu$,
so that $\mu$ is characterized by its moments.
As a conclusion, $\mu_n$ goes weakly, as $n\to\infty$, to $\mu$. 
This shows point (ii) for all $\phi:[0,\infty)\mapsto \rr$ continuous 
and bounded. The extension to continuous functions with at most polynomial
growth easily follows from point (i). 
\end{proof}

Let us now show that $\mu$ satisfies some equilibrium equation.

\begin{lem}\label{lequi}
Consider the probability measure $\mu$ on $[0,\infty)$
defined in Lemma \ref{tlm}. For all $\phi \in C^1([0,\infty))$ such that 
$\phi$ and $\phi'$ have
at most polynomial growth,
\begin{align*}
&2|a_1'| \int_0^\infty \int_0^\infty \left[\frac{\phi(x+y)-\phi(x)}{y}
\indiq_{\{y>0\}}
+\phi'(x)\indiq_{\{y=0\}}\right]\mu(dx)\mu(dy) \\
&\hskip2cm= 
\int_0^\infty x[\phi(x)-\phi(0)]\mu(dx).
\end{align*}
\end{lem}

\begin{proof}
We consider $\phi$ as in the statement and, for $n\in \nn$,
$\mu_n=\sum_{k\geq 1}kc^n_k\delta_{n^{-2/3}k}$ as in
Lemma \ref{tlm}. Apply (\ref{wf2}) with $\psi(k)=\phi(n^{-2/3}k)$:
\begin{align*}
&\frac{2n^{-2/3}}{m_0(c^n)} \sum_{k,l\geq 1} 
\frac{\phi(n^{-2/3}(k+l))-\phi(n^{-2/3}k)}{n^{-2/3}l}
k c^n_k lc^n_l\\
=& \frac {n^{2/3}} n \sum_{k \geq 1} \left(\phi(n^{-2/3}k)-\phi(n^{-2/3}) \right) 
(n^{-2/3}k-n^{-2/3}) k c^n_k.
\end{align*}
Multiply this equality by $n^{1/3}$. 
In terms of $\mu_n$, this can be written as
\begin{align*}
&\frac{2n^{-1/3}}{m_0(c^n)}
\int_0^\infty\int_0^\infty \frac{\phi(x+y)-\phi(x)}{y}
\mu_n(dx)\mu_n(dy)\\
=& \int_0^\infty [\phi(x)-\phi(n^{-2/3})](x-n^{-2/3})\mu_n(dx).
\end{align*}
Recall that $m_0(c^n)\sim n^{-1/3}/|a_1'|$ as $n \to \infty$, so that
\begin{align*}
&2|a_1'| \int_0^\infty\int_0^\infty \frac{\phi(x+y)-\phi(x)}{y}
\mu_n(dx)\mu_n(dy)\\
\sim&  
\int_0^\infty [\phi(x)-\phi(n^{-2/3})](x-n^{-2/3})\mu_n(dx)\\
\sim& \int_0^\infty [\phi(x)-\phi(0)] x\mu_n(dx).
\end{align*}
To obtain the last equivalent, use that $\phi$ is continuous, that
$\sup_n \int_0^\infty x\mu_n(dx)<\infty$ and that $\sup_n \int_0^\infty
|\phi(x)-\phi(0)| \mu_n(dx)<\infty$ by Lemma \ref{tlm}-(i) since
$\phi$ has at most polynomial growth.
Define now the function $\Gamma(x,y)=([\phi(x+y)-\phi(x)]/y) \indiq_{\{y>0\}}
+\phi'(x)\indiq_{\{y=0\}}$, which is continuous and has at most polynomial growth
on $[0,\infty)^2$. Since $\mu_n$ does not give weight to $0$, we have
\begin{align}\label{tpttl}
&2|a_1'| \int_0^\infty\int_0^\infty \Gamma(x,y)
\mu_n(dx)\mu_n(dy) \\
=& 2|a_1'| \int_0^\infty\int_0^\infty \frac{\phi(x+y)-\phi(x)}{y}
\mu_n(dx)\mu_n(dy)  \nonumber\\
\sim& \int_0^\infty [\phi(x)-\phi(0)] x\mu_n(dx). \nonumber
\end{align}
Recall that for any $\psi\in C([0,\infty))$ with at most polynomial 
growth,
there holds
$\int_0^\infty \psi(x)\mu_n(dx) \stackrel {n\to \infty} \to
\int_0^\infty \psi(x)\mu(dx)$
due to Lemma \ref{tlm}.
This implies that for all $\Psi\in C([0,\infty)^2)$ 
with at most polynomial growth,
$\int_0^\infty\int_0^\infty \Psi(x,y)\mu_n(dx)\mu_n(dy)
\stackrel {n\to \infty} \to
\int_0^\infty\int_0^\infty \Psi(x,y)\mu(dx)\mu(dy)$.
Taking the limit as $n\to\infty$ in (\ref{tpttl}), we deduce that
\begin{align*}
2|a_1'| \int_0^\infty\int_0^\infty \Gamma(x,y)
\mu(dx)\mu(dy)=\int_0^\infty [\phi(x)-\phi(0)] x\mu(dx)
\end{align*}
as desired.
\end{proof}

We now try to determine a quantity resembling the Laplace transform of $\mu$.
This is a usual trick for Smoluchowski's coagulation with constant kernel,
see e.g. Deaconu-Tanr\'e \cite{dt}.

\begin{lem}\label{laplace}
Consider the probability measure $\mu$ defined in Lemma \ref{tlm}.
Then for all $q\geq 0$,
\begin{align*}
\ell(q):=\int_0^\infty (1-e^{-qx})\frac{\mu(dx)}x = \frac{-\Ai'(q+a_1')}
{|a_1'|\Ai(q+a_1')}.
\end{align*}
\end{lem}

\begin{proof} We put $\beta=\mu(\{0\})$, which we cannot exclude to be
nonzero at the moment.
We apply Lemma \ref{lequi} with $\phi(x)=(1-e^{-qx})/x$,
which is indeed in $C^1([0,\infty))$ with $\phi(0)=q$ and $\phi'(0)=-q^2/2$.
This yields
\begin{align*}
2|a_1'|[A_1(q)+A_2(q)+A_3(q)+A_4(q)]=B_1(q)-B_2(q),
\end{align*}
where
\begin{align*}
A_1(q)=&\int_0^\infty\int_0^\infty\frac{\phi(x+y)-\phi(x)}{y}\indiq_{\{x>0,y>0\}}
\mu(dx)\mu(dy),\\
A_2(q)=&\beta \int_0^\infty \frac{\phi(y)-\phi(0)}{y}\indiq_{\{y>0\}}\mu(dy),\\
A_3(q)=&\beta \int_0^\infty \phi'(x) \indiq_{\{x>0\}}\mu(dx),\\
A_4(q)=&\beta^2 \phi'(0),\\
B_1(q)=&\int_0^\infty x\phi(x) \mu(dx),\\
B_2(q)=& \phi(0) \int_0^\infty x\mu(dx). 
\end{align*}
Recalling Lemma \ref{tlm} and that $\phi(0)=q$, 
we see that $B_2(q)=q /|a_1'|$. Next,
\begin{align*}
B_1(q)=\int_0^\infty (1-e^{-qx})\mu(dx) = 1- \ell'(q).
\end{align*}
One can check that $\phi'(0)=-q^2/2$, whence $A_4(q)=-\beta^2q^2/2$.
A computation shows that $[\phi(x)-\phi(0)]/x+\phi'(x)=-q\phi(x)$, so that
\begin{align*}
A_2(q)+A_3(q)=- \beta q \int_0^\infty x^{-1}(1-e^{-qx})\indiq_{\{x>0\}}\mu(dx)
= - \beta q [\ell(q)- \beta q].
\end{align*}
Finally, using a symmetry argument and then that
$(x+y)\phi(x+y)-x\phi(x)-y\phi(y)= - xy\phi(x)\phi(y)$, 
\begin{align*}
A_1(q)=& \int_0^\infty \int_0^\infty \frac{x\phi(x+y)-x\phi(x)}{xy}
\indiq_{\{x>0,y>0\}}\mu(dx)\mu(dy)\\
=& \frac 1 2 \int_0^\infty \int_0^\infty \frac{(x+y)\phi(x+y)-x\phi(x)
-y\phi(y)}{xy} \indiq_{\{x>0,y>0\}} \mu(dx)\mu(dy) \\
=& - \frac 1 2 \left( \int_0^\infty (1-e^{-qx})\indiq_{\{x>0\}}\frac{ \mu(dx)}x 
\right)^2\\
=& - \frac 1 2 \left( \ell(q)-\beta q \right)^2.
\end{align*}
All this shows that
\begin{align*}
2|a_1'|\left[-\frac 1 2 \left( \ell(q)-\beta q \right)^2
- \beta q [\ell(q)- \beta q] -\beta^2q^2/2 \right] = 1- \ell'(q) - q/|a_1'|,
\end{align*}
whence
\begin{align}\label{equadiff}
\ell'(q)=1-q/|a_1'|+|a_1'|\ell^2(q).
\end{align}
This equation, together with the initial condition $\ell(0)=0$,
has a unique maximal solution due to the Cauchy-Lipschitz Theorem.
And one can check that this unique maximal solution is nothing but
\begin{align*}
\ell(q)= \frac{-\Ai'(q+a_1')}{|a_1'| \Ai(q+a_1')},
\end{align*}
which is defined for $q\in (a_1-a_1',\infty)$, because the Airy function does
not vanish on $(a_1,\infty)$. Indeed, it suffices to use that since
$\Ai''(x)=x\Ai(x)$, 
\begin{align*}
\frac{d}{dx}\left(\frac{\Ai'(x)}{\Ai(x)} \right)= x - 
\left(\frac{\Ai'(x)}{\Ai(x)} \right)^2
\end{align*}
and that by definition, $\Ai'(a_1')=0$.
\end{proof}

We now write down two formulae of Darling \cite{d} and Louchard \cite{l} 
that we found in the survey paper
of Janson \cite[p 94]{j}.
Denote by $(\be_t)_{t\in [0,1]}$ the normalized Brownian excursion
and define its area as $\bex=\int_0^1 \be_t dt$. Put, for $y\geq 0$,
\begin{align}\label{psiex}
\psiex(y)=\E[e^{-y \bex}]. 
\end{align}
There hold
\begin{align}\label{ex1}
&\psiex(y) = \sqrt{2\pi} y \sum_{j=1}^\infty \exp\left(-2^{-1/3}|a_j|y^{2/3} \right),
\\
\label{ex2}
&\int_0^\infty (1-e^{-qy}) \frac{\psiex(y^{3/2})}{\sqrt{2\pi y^3}}dy= 
2^{1/3}\left(\frac{\Ai'(0)}{\Ai(0)}-\frac{\Ai'(2^{1/3}q)}{\Ai(2^{1/3}q)}\right).
\end{align}
This allows us to find a link between our probability measure $\mu$
and $\psiex$.

\begin{lem}\label{lastlem}
Consider the probability measure $\mu$ defined in Lemma \ref{tlm}.
Then 
\begin{align*}
\frac{\mu(dx)}{x}= \frac{\psi_{\rm ex}(\sqrt 2 x^{3/2})e^{|a_1'|x} }
{2\sqrt \pi |a_1'| x^{3/2}} \indiq_{\{x>0\}} dx.
\end{align*}
\end{lem}

\begin{proof} 
Using Lemma \ref{laplace}, we deduce that for all $q\geq 0$,
\begin{align*}
\int_0^\infty (1-e^{-qx})\frac{e^{-|a_1'|x}\mu(dx)}{x}=
&\int_0^\infty (e^{-|a_1'|x}-1)\frac{\mu(dx)}{x} + 
\int_0^\infty (1-e^{-(|a_1'|+q)x})\frac{\mu(dx)}{x}\\
=&-\ell(|a_1'|) + \ell(q+|a_1'|)\\
=&\frac{1}{|a_1'|}\left(\frac{\Ai'(0)}{\Ai(0)}-\frac{\Ai'(q)}{\Ai(q)} \right).
\end{align*}
Next, (\ref{ex2}) implies that for all $q\geq 0$,
\begin{align}\label{for}
\int_0^\infty (1-e^{-qx}) \frac{\psiex(\sqrt 2 x^{3/2})}{2\sqrt{\pi}|a_1'|
x^{3/2}}dx
=& \int_0^\infty (1-e^{-q2^{-1/3}y}) \frac{\psiex(y^{3/2})}
{|a_1'|\sqrt{2\pi y^3}} 2^{-1/3}dy
\\
=&\frac 1 {|a_1'|}\left(\frac{\Ai'(0)}{\Ai(0)}-\frac{\Ai'(q)}{\Ai(q)}\right).\nonumber
\end{align}
We conclude by injectivity of the Laplace transform. 
\end{proof}

We may finally give the

\begin{preuve} {\it of Theorem \ref{mr3}.}
Define $c:(0,\infty)\mapsto(0,\infty)$ by
\begin{align}\label{sss1}
c(x)= \frac{\psiex(\sqrt 2 x^{3/2})e^{|a_1'|x}}{
2\sqrt \pi|a_1'| x^{3/2}}.
\end{align}
By Lemma \ref{lastlem}, the probability measure $\mu$
defined in Lemma \ref{tlm} is nothing but $\mu(dx)=xc(x)dx$. 
We thus have $\int_0^\infty xc(x)dx=1$. 
Recalling Lemma \ref{tlm}, we know that 
$\lim_n \sum_{k\geq 1} \phi(n^{-2/3}k)kc^n_k=\int_0^\infty \phi(x)xc(x)dx$ 
for any $\phi\in C([0,\infty))$
with at most polynomial growth.
Using (\ref{ex1}), we immediately deduce that 
\begin{align}\label{sss2}
c(x)= |a_1'|^{-1} e^{|a_1'|x} \sum_{j=1}^\infty e^{-|a_j|x}.
\end{align}
It is clear from  (\ref{sss2}) that $c \in C^\infty((0,\infty))$ and that
$c(x) \stackrel \infty \sim |a_1'|^{-1} e^{(|a_1'|-|a_1|)x}$: it suffices to use 
that $0<|a_1|<|a_2|<\dots$, that $|a_j|\stackrel {j\to \infty}
\sim (3\pi j/2)^{2/3}$
and Lebesgue's dominated convergence theorem. It is immediate from
(\ref{sss1}) that  $c(x) \stackrel 0 \sim 1/(2\sqrt \pi |a_1'| x^{3/2})$.
Since we now know that $\mu(dx)=c(x)dx$ does not give weight to zero, (\ref{eqlim}) follows from 
Lemma \ref{lequi}.
Finally, (\ref{lala}) follows from Lemma \ref{laplace} when $q\geq 0$. It is easily
extended to $q>a_1-a_1'$ using that $c(x) \stackrel \infty \sim |a_1'|^{-1} e^{(|a_1'|-|a_1|)x}$.
\end{preuve}

\section{Perfect simulation algorithms}\label{proba}

In this section, we provide some perfect simulation algorithms: 
we introduce a pruning procedure $\cP_n$ of trees (for each $n\geq 1$) that will allow us to interpret

\vip

$\bullet$ the particles'  mass distribution $(p^n_k)_{k\geq 1}$ as the law of the number of leaves of
$G_n=\cP_n(G)$, where $G$ is a binary critical Galton-Watson tree;

\vip

$\bullet$ the size-biased particles' mass distribution 
$(kc^n_k)_{k\geq 1}$ as the law of the number of leaves of $\hG_n=\cP_n(\hG)$, where $\hG$ is a
size-biased binary critical Galton-Watson tree.

\vip

We also show that $G_n$ obviously tends to $G$  as $n\to\infty$, which gives a probabilistic interpretation 
of Theorem \ref{mr2}.

\subsection{The pruning procedure}\label{secpn}

Consider a rooted discrete binary tree $T$, 
that is a set of vertices $V(T)$
and of edges $E(T)$, satisfying the usual properties of binary trees. 
The root is denoted by $\emptyset$ and we add a vertex $\star$ and an edge joining
$\star$ to $\emptyset$. 
We denote by $L(T)$ the set of the leaves of $T$. We take the convention that
$\star$ is not a leave, but $\emptyset$ may be a leave (iff $V(T)$ is reduced to $\{\star,\emptyset\}$).
We denote by $N(T)$ the set of internal vertices (nodes) of $T$, that is
$N(T)=V(T)\setminus (\{\star\}\cup L(T))$.

\vip

We recall that for any binary tree $T$, it holds that $|L(T)|=|N(T)|+1$.

\begin{defin}\label{pn}
Let $T$ be a rooted discrete binary tree with at most one infinite branch and let $n\geq 1$.
We define the (random) subtree $\cP_n(T)$ as follows.

\vip

{\it Step 1.}  Endow each edge $e$ with the length 
$\kappa_e$, where $(\kappa_e)_{e\in E(T)}$ is an i.i.d. family of $\expo(2)$-distributed random variables.
This induces a distance $d$ on $T$. For each $t\geq 0$, define $T_t=\{x\in T : d(x,\star)=t\}$
and let $t_0=\inf\{t>0 : T_t=\emptyset\}\leq \infty$ be the height of $T$.

\vip

{\it Step 2.} For each internal node $x\in N(T)$, consider the branch $B_T(x)$ (endowed with the lengths
introduced at Step 1) joining $x$ to $\star$ and consider a Poisson point process $\pi_x^n$ with
rate $1/n$ on $B_T(x)$ (conditionally on  $(\kappa_e)_{e\in E(T)}$, 
these Poisson processes are taken mutually independent).
All the marks of these Poisson process are activated.

\vip

{\it Step 3.} We explore the tree from the top until we arrive at $\star$ 
(that is, we consider $T_t$ for $t$ decreasing from $t_0$ to $0$) 
with the following rule: each time we encounter an active mark of a
Poisson point process (defined in Step 2), 
we remove all the subtree above the mark (and replace it by a leave) 
and we deactivate all the marks of the Poisson point processes corresponding 
to nodes in this subtree. See Figure \ref{fig1} for an illustration.

\vip

{\it Step 4.} We call $\cP_n(T)$ the resulting tree, which is a subtree of $T$ containing
$\star$.

\vip

A mark of a Poisson process is said to be {\rm useful} if has generated a pruning at some
time in the procedure and {\rm useless} otherwise.
\end{defin}

\begin{figure}[hb]
\fbox{
\begin{minipage}[c]{0.95\textwidth}
\centering
\resizebox{0.47\linewidth}{!}{\input{Prune1.pstex_t}} 
\resizebox{0.47\linewidth}{!}{\input{Prune2.pstex_t}} \\
\resizebox{0.47\linewidth}{!}{\input{Prune3.pstex_t}} 
\resizebox{0.47\linewidth}{!}{\input{Prune4.pstex_t}} 
\caption{The pruning procedure $\cP_n$}
\label{fig1}
\vip
\parbox{12cm}{
\footnotesize{
On Figure A, the tree $T$ is drawn, its edges being endowed with i.i.d. $\expo(2)$-distributed random
variables. The marks of the Poisson processes are represented as follows: here $\pi_x^n$ has one mark (on
the edge just under $x$), $\pi_y^n$ has one mark (on the edge above $\emptyset$ on the left),
$\pi_z^n$ has three marks (one on the edge just under $z$, one on the  
edge above $\emptyset$ on the right, one on the edge $(\star,\emptyset)$) and
$\pi_t^n$ has one mark on the edge just under $t$. All the other Poisson point processes have no mark.

Thus starting from the top of the tree, we first encounter the mark just under $z$. On Figure
B, we have drawn the resulting tree: we have replaced the subtree above this mark by a leave
and we have erased (deactivated) 
the marks {\it sent} by nodes of this subtree (here, the two other marks sent by $z$).

Then we encounter the mark sent by $t$ and the resulting tree is drawn on Figure C. Finally,
we encounter the mark sent by $x$, which makes inactive the mark sent by $y$ and the resulting
tree is drawn on figure D. Since there are no marks any more, the tree of figure D is 
$\cP_n(T)$.
}}
\end{minipage}
}
\end{figure}

If $T$ is infinite, then Step 3 looks ill-posed at first glance, since the top of the tree
lies at infinity.

\begin{rem}\label{mqm}
Let $T$ be a tree with one infinite branch. Then Definition \ref{pn} makes sense.
Indeed, consider the first (starting from $\star$) edge $e=(x_1,x_2)$ on this infinite branch, 
with $x_2$ child of $x_1$, such that $\pi_{x_2}^n$ has a mark in $e$. 
This happens for each edge independently with positive probability (not depending on $e$), 
so that such an edge a.s. exists.
Then $e$ will be cut (either by $\pi_{x_2}^n$ or by a Poisson process corresponding to a node above $x_2$)
and this will make inactive all the Poisson processes corresponding to nodes above $x_2$.
Thus everything will happen as if $x_2$ was a leave of $T$, neglecting all the marks
due to Poisson processes corresponding to nodes above $x_2$. See Figure \ref{fig2} for an illustration.
\end{rem}

\begin{figure}[hb]
\fbox{
\begin{minipage}[c]{0.95\textwidth}
\centering
\resizebox{0.47\linewidth}{!}{\input{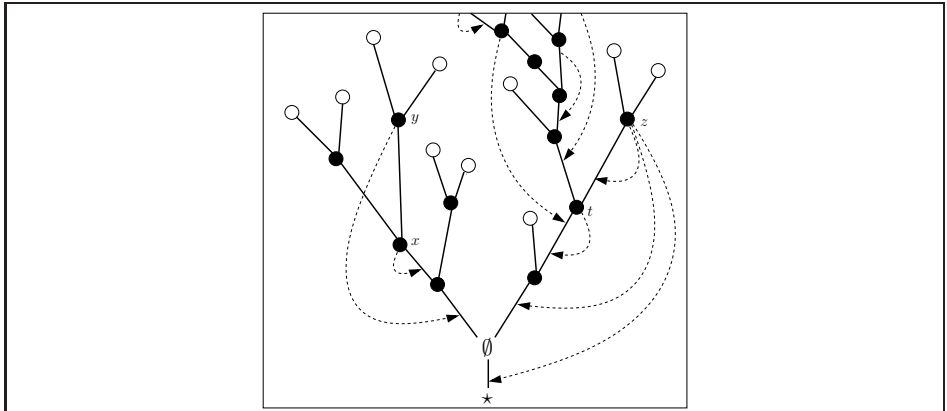}} 
\caption{Applying $\cP_n$ to a tree with one infinite branch}
\label{fig2}
\vip
\parbox{12cm}{
\footnotesize{
The first edge $e=(x_1,x_2)$ on the infinite branch 
such that $\pi^n_{x_2}$ has a mark on $e$ is here the edge under $t$ and $x_2=t$.
Then independently of what can happen above $t$, this edge will be cut and 
all the marks {\it sent} by nodes above $t$ will be erased. Indeed, there 
are two possibilities, calling 
$M$ the mark sent by $t$. 

Case 1: This mark $M$ is useful and thus $M$ is replaced by a leave and
all the marks sent between $\star$ and $M$ by nodes above $t$ are erased.

Case 2. This mark $M$ is useless and then it is necessarily deactivated by a useful mark
lying between $M$ and $t$ (sent by a node above $t$). We conclude as in case 1.

In any case, it is not necessary to know what happens above $t$ to conclude that, 
with this configuration, $\cP_n(T)$ will be the same as in Figure \ref{fig1}-D.
}}
\end{minipage}
}
\end{figure}

\subsection{Particles' mass distribution}

We can now give an interpretation in terms of trees of the particles' mass distribution.

\begin{prop}\label{treesOK}
Consider a binary critical Galton-Watson tree $G$ (BCGWT in short), that is a Galton-Watson tree
with offspring distribution $(\delta_0+\delta_2)/2$. 
Fix $n\geq 1$ and let $G_n=\cP_n(G)$, see Definition \ref{pn} 
(all the random objects used by $\cP_n$ are taken conditionally on $G$).

(i) For all $k\geq 1$,
$\Pr[|L(G)|=k]=p_k$, where $p_k$ was defined in Theorem \ref{mr2}.

(ii) For all $k\geq 1$, $\Pr[|L(G_n)|=k]=p^n_k$, where
$p^n_k$ was defined in Theorem \ref{mr2}.
\end{prop}

Clearly, $G_n=\cP_n(G)$ can be perfectly simulated. We thus have a perfect simulation algorithm
for $(p^n_k)_{k\geq 1}$. 

\vip

{\it Heuristic proof.} 
Any particle with mass $k$ can be seen
a cluster of $k$ particles with mass $1$. And it is natural to use a
tree to represent the {\it genealogy} of this particle; if this
particle has a mass $k$, then this tree will have $k$ leaves.

To be more precise, we need first to handle some computations. 
Divide
(\ref{wf1}) by $m_0(c^n)$, use that $\sum_{k,l\geq 1} \phi(l)p^n_k p^n_l=\sum_{k\geq 1} \phi(k)p^n_k$
and that, recalling (\ref{ettac2}), 
$$
\frac1n\sum_{k\geq 2}(k-1)^2 \phi(1)p^n_k=\phi(1)\frac{m_2(c^n )+m_0(c^n)-2m_1(c^n)}{nm_0(c^n)}
=\phi(1)= \sum_{k\geq 1} \phi(1)p^n_k.
$$ 
One gets, for all reasonable $\phi$,
$$
\sum_{k,l\geq 1} [\phi(k+l)-\phi(k)]p^n_k p^n_l + \sum_{k\geq 2} [\phi(1)-\phi(k)] [1+(k-1)/n]p^n_k=0.
$$
Thus $(p^n_k)_{k\geq 1}$ can be seen as the equilibrium of the mass of a particle with
the following dynamics: (i) it merges with an independent similar particle at rate $1$,
(ii) its mass is reset to $1$ at rate $1$, (iii) 
its mass is reset to $1$ at rate $(k-1)/n$, where $k$ is its mass.
 
Consider now such a particle at equilibrium. 
First neglect (iii) and follow the history of the particle backward
in time: it has merged with a similar particle at rate $1$
and it has been reset to $1$ at rate $1$.
Thus this particle is subjected to events at rate $2$ and each time
an event occurs, it is a coalescence (node) with probability $1/2$ and a breakage 
(leave) with probability $1/2$.
This can be represented by a BCGWT, of which the edges have a length
with law $\expo(2)$. Now we take (iii) into account: we start from the past (i.e. from the top of the tree),
we follow the branches of the tree and reset the particle to $1$ at rate $(k-1)/n$,
where $k$ is the mass of the particle, i.e. the number of leaves of the subtree above the point
under consideration. Using finally that the number of nodes
of a binary tree is precisely its number of leaves minus $1$, we guess that $\cP_n(G)$ should indeed 
provide a perfect simulation algorithm for the genealogy of a particle at equilibrium.

\vip

We now handle a rigorous proof.

\begin{proof}
We start with (i), which is completely standard. Define $q_k=\Pr[|L(G)|=k]$ for $k\geq 1$ and use the 
{\it branching property}: knowing that $G$ is not reduced to the root, it can be written
as $G=\{\emptyset\}\cup G'\cup G''$, for two independent copies $G',G''$ of $G$.
Hence conditionally on $\{|L(G)| \geq 2\}$,
$L(G)=L(G')\cup L(G'')$, whence, for $k \geq 2$, 
$\{|L(G)|=k\}=\{|L(G)|\geq 2\}\cap \cup_{i=1}^{k-1}\{|L(G')|=i,|L(G'')|=k-i\}$
and thus $q_k = \frac 12 \sum_{i=1}^{k-1} q_iq_{k-i}$. 
Next, it obviously holds that
$q_1=1/2$. Since $p_1=q_1$ and recalling (\ref{limsimple}), it follows that $q_k = p_k$ for all $k\geq 1$.

\vip

We now check (ii). Using the branching property of $G$ recalled above,
it immediately follows from the pruning procedure $\cP_n$ that for $k\geq 2$,
$$
\{|L(G_n)|=k\} = \{|L(G)|\geq 2\} \cap \bigcup_{i=1}^{k-1}\{|L(G_n')|=i,
|L(G_n'')|=k-i,A_n\},
$$ 
where $A_n$ is the event that no pruning occurs in the edge $(\star,\emptyset)$
and where $G_n'$ and $G_n''$ are two independent copies of $G_n$, independent
of $(\{L(G)\geq 2\},(\pi_x^n\vert_{(\star,\emptyset)})_{x \in N(G)})$. 
First observe that conditionally on $(G'_n,G''_n)$,
the event $A_n$ occurs if for all $x\in N(G'_n)\cup N(G''_n)\cup \{\emptyset\}$,
$\pi^n_x$ has no mark in $(\star,\emptyset)$. Recalling that the length of 
this edge is $\expo(2)$-distributed, 
that the Poisson processes $\pi^n_x$
have rate $1/n$ and that
here we have $|N(G'_n)\cup N(G''_n)\cup \{\emptyset\}|=|N(G'_n)|+|N(G''_n)|+1=
|L(G'_n)|+|L(G''_n)|-1$ Poisson processes, one easily deduces that 
$$
\Pr[A_n\;\vert\; G_n',G_n'']
=\frac{2}{(|L(G'_n)|+|L(G''_n)|-1)/n+2}.
$$
Put now $r^n_k=\Pr[|L(G_n)|=k]$. From the previous study, we get, for $k\geq 2$,
$$
r^n_k = \frac 1 2 \sum_{i=1}^{k-1} r^n_i r^n_{k-i} \frac{2}{(i+k-i-1)/n+2},
$$
whence $[2+(k-1)/n]r^n_k=\sum_{i=1}^{k-1} r^n_i r^n_{k-i}$.
Using the arguments and notation of the proof of Lemma \ref{equi}-Step 1
(with $r=1$), 
we deduce that $r^n_k=\alpha^n_k (r_1^n)^k$ for all $k\geq 1$. 
But we also know, since $G_n\subset G$ is a.s. finite, that
$\sum_{k\geq 1} \alpha^n_k (r_1^n)^k= \sum_{k\geq 1} r^n_k=1$. Recalling
(\ref{generaf}) and Lemma \ref{exiq}, we conclude that 
$r^n_1=q_n$, whence $r^n_k=\alpha^n_k (q_n)^k$ for all $k\geq 1$.
By Lemma \ref{equi}, it also holds that $p^n_k=c^n_k/m_0(c^n)=\alpha^n_k (q_n)^k$,
which concludes the proof.
\end{proof}

\subsection{Probabilistic interpretation of Theorem \ref{mr2}}
This is not hard from Proposition \ref{treesOK}. 
Clearly, as $n\to \infty$, the probability that $G=G_n$ tends
to $1$, implying that the law of $|L(G_n)|$, i.e. $(p^n_k)_{k\geq 1}$, tends to the law of
$|L(G)|$, i.e.  $(p_k)_{k\geq 1}$.

Indeed, we have $G_n=G$ as soon as for all $x \in N(G)$, 
$\pi^n_x$ has no mark on the branch $B_G(x)$ joining $\star$ to $x$ in $G$. 
Conditionally on $(G,(\kappa_e)_{e\in N(G)})$, this occurs with probability 
$\prod_{x \in N(G)} \exp\left(-\frac1n\sum_{e\in E(G), e\subset B_G(x)} \kappa_e\right)$, which a.s. tends to $1$,
since $G$ is a.s. a finite tree.

\subsection{Size-biased particles' mass distribution}

We now interpret the size-biased particles' mass distribution in 
terms of a pruned size-biased Galton-Watson.

\begin{defin}\label{fgw}
Consider a family of i.i.d. binary critical Galton-Watson trees $(G(i))_{i\geq 1}$.
We call {\rm size-biased binary critical Galton-Watson tree} (SBBCGWT in short) 
the binary tree $\hG$ with one infinite branch (called the backbone)
as in Figure \ref{fig3}, where for each $i\geq 1$, we plant $G(i)$ on $\star i$.
\end{defin}

\begin{figure}[hb]
\fbox{
\begin{minipage}[c]{0.95\textwidth}
\centering
\resizebox{0.66\linewidth}{!}{\input{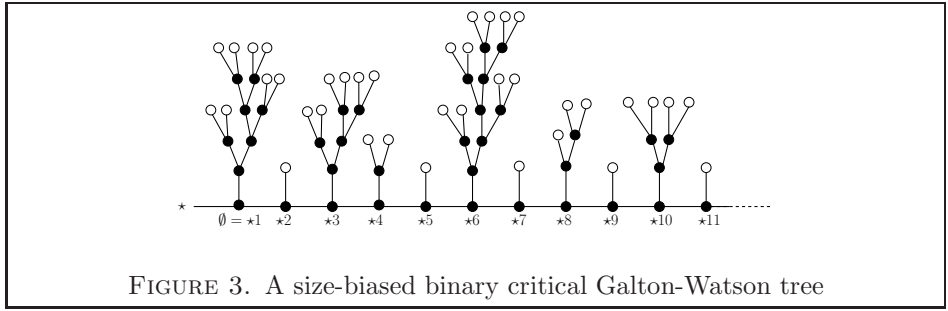}} 
\caption{A size-biased binary critical Galton-Watson tree}
\label{fig3}
\end{minipage}
}
\end{figure}

Traditionally, each $G(i)$ is planted above or under the backbone at random, but this
is absolutely useless for our purpose, since we never take into account any order on the vertices.
See Lyons-Pemantle-Peres \cite[Section 2]{lpp} for some indications about the terminology
{\it size-biased}.

\begin{prop}\label{treesOK2}
Consider a SBBCGWT $\hG$ as in Definition 
\ref{fgw} and fix $n\geq 1$. Let $\hG_n=\cP_n(\hG)$, recall Definition \ref{pn} 
(all the the random objects used by $\cP_n$ are taken conditionally on $\hG$). Then for all $k\geq 1$,
$\Pr(|L(\hG_n)|=k)=kc^n_k$, where $(c^n_k)_{k\geq 1}$ was defined in Theorem
\ref{mr1}.
\end{prop}

Since $\cP_n(\hG)$ can be perfectly simulated due to Remark \ref{mqm},
this provides a perfect simulation algorithm
for $(kc^n_k)_{k\geq 1}$.
It seems striking that $\cP_n(\hG)$ is a size-biased
version of $\cP_n(G)$, with $\hG$ a size-biased version of $G$.
However, the two notions of {\it size-biased} are quite different:
$\cP_n(\hG)$ is a version of $\cP_n(G)$ {\it biased by the number of leaves},
while $\hG$ is a version of $G$ {\it biased by the size of the population at generation} 
$n$ (with $n\to \infty$).

\vip

{\it Heuristic proof.} 
First rewrite (\ref{wf2}) as 
$$
\sum_{k,l\geq 1} [\psi(k+l)-\psi(k)]kc^n_k 2p^n_l + \sum_{k\geq 2} [\psi(1)-\psi(k)] \frac{k-1}n kc^n_k=0.
$$
Thus $(kc^n_k)_{k\geq 1}$ can be seen as the equilibrium of the mass of a particle with
the following dynamics: (i) it merges with an independent particle with law $(p^n_k)_{k\geq 1}$ at rate $2$,
(ii) its mass is reset to $1$ at rate $(k-1)/n$, where $k$ is its mass.
 
Consider such a particle at equilibrium and first neglect (ii). Then obviously, we can 
represent its genealogy as a forest of (non size-biased) particles $\cP_n(G(i))$,
the length of the edges one the backbone (those between these particles) being $\expo(2)$-distributed.
Then, take (ii) into account: start from the past and prune the edges on the backbone
at rate $(k-1)/n$, where $k$ is the mass of the particle, i.e. the number of leaves of the subtree
above the point under consideration, whence $k-1$ is the corresponding number of nodes.

What we really do is slightly different, since we prune the (non size-biased) particles 
and the backbone simultaneously, but one can easily get convinced that this changes nothing.

\vip

We now give some rigorous arguments.

\begin{proof}
Consider the problem with unknown $(t^n_k)_{k\geq 1}$ (here $(p^n_k)_{k\geq 1}$ is given)
\begin{equation}\label{tasoeur}
\sum_{k\geq 1} t^n_k=1, \quad [(k-1)/n+2]t^n_k=2\sum_{i=1}^{k-1}p^n_i t^n_{k-i} 
\quad (k\geq 2).
\end{equation}

{\it Step 1.} Define $s^n_k=\Pr(|L(\hG_n)|=k)$ for $k\geq 1$. We show here
that $(s^n_k)_{k\geq 1}$ is a solution to (\ref{tasoeur}).
First recall from Remark \ref{mqm} that $\hG_n$ is a.s. finite, 
whence $\sum_{k\geq 1} s^n_k=1$.  
We introduce $G(1)$ the BCGWT planted on $\star 1$ and
$\hG'$ the SBBCGWT on the right of $\star 1$
(see Figure \ref{fig3}). Clearly, for $k\geq 2$, we can write
$$
\{|L(\hG_n)|=k\}=\cup_{i=1}^k \{|L(G_n(1))|=i,|L(\hG_n')|=k-i\}\cap A_n,
$$
where $A_n$ is the event that there is no pruning in the edge $(\star,\star 1)$,
where $G_n(1)=\cP_n(G(1))$ and $\hG'_n=\cP_n(\hG')$. Note that $G(1)$ and $\hG'$ 
are independent and pruned independently, so that $G_n(1)$ and $\hG'_n$ are independent. 
We know by Proposition \ref{treesOK} that
$|L(G_n(1))|$ is $(p^n_k)_{k\geq 1}$-distributed. Furthermore, it is clear
that $|L(\hG_n')|$ has the same law as $|L(\hG_n)|$. Finally, exactly as 
in the proof of Proposition \ref{treesOK}, one may check that
$$
\Pr[A_n\;\vert\; G_n(1),\hG_n']
=\frac{2}{(|L(G_n(1))|+|L(\hG_n')|-1)/n+2}.
$$
As a conclusion, there holds
$$
s^n_k=\sum_{i=1}^{k-1} p^n_i s^n_{k-i} \frac{2}{(i+k-i-1)/n+2},
$$
whence $[(k-1)/n+2]s^n_k=2\sum_{i=1}^{k-1} p^n_i s^n_{k-i}$ as desired.
\vip

{\it Step 2.} Next, we show that $(kc^n_k)_{k\geq 1}$ also solves 
(\ref{tasoeur}). Recall that $\sum_{k\geq 1} k c^n_k=1$ due to 
(\ref{eq0}). For $k\geq 2$, using (\ref{eq2}),
\begin{align*}
[(k-1)/n+2]kc^n_k=&\sum_{i=1}^{k-1} \frac{(i+(k-i))c^n_i c^n_{k-i}}{m_0(c^n)}
=2\sum_{i=1}^{k-1} \frac{(k-i)c^n_i c^n_{k-i}}{m_0(c^n)},
\end{align*}
so that $[(k-1)/n+2]kc^n_k=2\sum_{i=1}^{k-1}p^n_i (k-i)c^n_{k-i}$.

\vip

{\it Step 3.} To conclude the proof, it remains to prove that 
(\ref{tasoeur}) has at most one solution.
For each $x>0$, there is obviously 
a unique solution $(u^n_k(x))_{k\geq 1}$ to 
$$
u^n_1(x)=x, \quad [(k-1)/n+2]u^n_k(x)=2\sum_{i=1}^{k-1}p^n_i u^n_{k-i}(x) \quad (k\geq 2).
$$
One immediately checks recursively that for each $k\geq 1$, 
$x\mapsto u^n_k(x)$ is increasing. Hence, there is at most one value $x_n>0$
such that $\sum_{k\geq 1} u^n_k(x_n)=1$. But any solution $(t^n_k)_{k\geq 1}$ 
to (\ref{tasoeur}) has to satisfy $(t^n_k)_{k\geq 1}=(u^n_k(t^n_1))_{k\geq 1}$
and thus
must be equal to $(u^n_k(x_n))_{k\geq 1}$.
\end{proof}

\section{A scaling limit for the size-biased typical particle}\label{proba2}

The aim of this section is to understand why the Brownian excursion arises in Theorem
\ref{mr3}.
We will build a random real tree $H$ and a pruning procedure $\cP_\infty$ such
that the {\it measure of leaves} of $H_\infty=\cP_\infty(H)$ follows the law $x c(x)dx$, where 
$c$ is the profile introduced in Theorem \ref{mr3}. We will also explain heuristically why $H_\infty$ 
should be the limit of $n^{-1/3}\hG_n$, where $\hG_n$ is the pruned SBBCGWT as in Proposition \ref{treesOK2},
thus providing an interpretation of Theorem \ref{mr3}. A rigorous proof of this scaling limit
could probably be handled, but this would be quite tedious and not very interesting.
For precisions about continuum random trees, see Aldous \cite{a1,a2,a3}, 
Duquesne-Le Gall \cite{dlg}, Le Gall \cite{lg1,lg2}.
Let us mention that some pruning procedures of continuum random trees have been introduced and studied
by Abraham-Delmas-Voisin \cite{adv}, but it is quite different from ours.

\vip

We proceed as follows: in Subsections \ref{ss1}, \ref{ss2} and \ref{ss3}, we recall basic facts about
real trees, Aldous's continuum random tree and Aldous's self-similar continuum random tree.
The pruning procedure is explained in Subsection \ref{ss4}, and we indicate heuristically
why the rescaled pruned SBBCGWT should converge to the pruned self-similar continuum random tree
in Subsection \ref{ss5}. Subsection \ref{ss6} is devoted to the proof that the number of leaves
of the pruned self-similar continuum random tree is $xc(x)dx$-distributed, where $c$ was defined in Theorem 
\ref{mr3}. 
In Subsection \ref{ss7}, we briefly indicate how the pruned self-similar continuum random tree
can be built in terms of contour functions. Finally, we discuss some remarkable property
of a related diffusion process in Subsection \ref{ss8}.

\subsection{Real trees}\label{ss1}

First recall the following definition, see e.g. Le Gall \cite{lg2}.

\begin{defin}\label{realtree}
1. A measured rooted real tree $\cT=(T,d,\mu)$ is a metric space $(T,d)$, with a distinguished 
element $\rho$ called root, endowed with a nonnegative measure $\mu$, 
satisfying the following properties.

(i) For all $x,y\in T$, there is a unique isometry $\phi_{x,y}:[0,d(x,y)]\mapsto T$ such that
$\phi_{x,y}(0)=x$ and $\phi_{x,y}(d(x,y))=y$.

(ii) For all  $x,y\in T$, all continuous injective map $h:[0,1]\mapsto T$ 
with $h(0)=x$ and $h(1)=y$, there holds
$h([0,1])=\phi_{x,y}([0,d(x,y)])$.

2. We denote by $\tT$ the set of all measured rooted real trees.

3. For $\cT=(T,d,\mu)\in \tT$, we denote by $L(T)=\{x\in T : T\setminus\{x\}$ is connected$\}$ 
the set of its leaves
\end{defin}

There is a classical way to build a 
compact real tree from a continuous nonnegative function, see Le Gall \cite{lg2}.

\begin{lem}\label{cont}
Let $a \in [0,\infty)$ and $g:[0,a]\mapsto \rr_+$ be continuous and satisfy $g(0)=0$.
For $s,t \in [0,a]$,
put $\tilde d_g(s,t)=g(s)+g(t)-2 \min_{u\in [s\land t,s\lor t]}g(u) $, which is a pseudo-distance on $[0,a]$.
We introduce the equivalence relation $s \sim^g t$ iff $\tilde d_g(s,t)=0$. 
Consider the quotient space $T_g=[0,a]/\sim^g$ and the canonical projection $\pi_g:[0,a]\mapsto T_g$.
The pseudo-distance
$\tilde d_g$ induces a distance $d_g$ on $T_g$. Define the root $\rho=\pi_g(0)$ 
and introduce the measure $\mu_g= \lambda \circ \pi_g^{-1}$ on $T_g$, where $\lambda$
is the Lebesgue measure on $[0,a]$.
Then $\cT_g=(T_g,d_g,\mu_g)$ is a measured rooted real tree in the sense of Definition \ref{realtree}.
The function $g$ is called the contour of $\cT_g$.
\end{lem}

This result is easily extended to build real trees with one infinite branch.

\begin{lem}\label{cont2}
Let  $g:\rr \mapsto \rr_+$ be continuous, satisfy $g(0)=0$ and $g(-\infty)=g(\infty)=\infty$. 
For $s,t \in \rr$ with $s<t$, we introduce
$$
\tcl s,t \tcr = \left\{ \begin{array}{l}
[s,t]  \quad \hbox{if  $0\leq s \leq t$ or $s\leq t \leq 0$,}\\
(-\infty,s]\cup [t,\infty) \quad \hbox{if $s < 0 <t$.}
\end{array}\right.
$$
For $s,t \in \rr$,
put $\tilde d_g(s,t)=g(s)+g(t)-2 \min_{u\in \tcl s\land t,s\lor t \tcr}g(u) $, which is a pseudo-distance on $\rr$.
We introduce the equivalence relation $s\sim^g t$ iff $\tilde d_g(s,t)=0$. 
Consider the quotient space $T_g=\rr/\sim^g $ and the canonical projection $\pi_g:\rr\mapsto T_g$.
The pseudo-distance
$\tilde d_g$ induces a distance $d_g$ on $T_g$. Define the root $\rho=\pi_g(0)$ 
and introduce the measure $\mu_g= \lambda \circ \pi_g^{-1}$ on $T_g$, where $\lambda$
is the Lebesgue measure on $\rr$.
Then $\cT_g=(T_g,d_g,\mu_g)$ is a measured rooted real tree in the sense of Definition \ref{realtree}.
The function $g$ is called the contour of $\cT_g$.
\end{lem}

\subsection{Aldous's CRT}\label{ss2}

The famous continuum random tree (CRT) introduced by Aldous \cite{a1,a2,a3} is the real tree
$\cT_{2\be}=(T_{2\be},d_{2\be},\mu_{2\be})$, where $\be=(\be_t)_{t\in[0,1]}$ is a normalized Brownian excursion.
The factor $2$ is unimportant and replaced by $1$ in many papers.
Here we need the following one-parameter family of CRTs.

\begin{defin}\label{dfcrt}
Let $\be=(\be_t)_{t\in[0,1]}$ be the normalized Brownian excursion. For $u>0$, 
we introduce the Brownian excursion $\be^u=(\be^u_s)_{s\in [0,u]}$ with length $u$, defined by
$\be^u_s= \sqrt u . \be_{s/u}$.
We call CRT$(u)$ the random rooted measured real 
tree $\cT_{\be^u}=(T_{\be^u},d_{\be^u},\mu_{\be^u})$ (recall Lemma \ref{cont}) and we denote
by $Q_u$ its law, which is a probability measure on $\tT$.
\end{defin}

The CRT is universal in that it is the scaling limit of any reasonable critical or sub-critical
Galton-Watson tree conditioned on being large, as shown by Aldous \cite{a3}. 
Let us now recall some properties of the CRT we will use later.

\begin{lem}\label{crt}
Let $u>0$ and let  $\cT_{\be^u}=(T_{\be^u},d_{\be^u},\mu_{\be^u})$ be a CRT$(u)$ with root $\rho$.

(i) $\mu_{\be^u}(T_{\be^u})=\mu_{\be^u}(L(T_{\be^u}))=u$.

(ii) $\int_{T_{\be^u}} d_{\be^u}(\rho,x) \mu_{\be^u}(dx) = \int_0^u \be^u_s ds = u^{3/2}\int_0^1 \be_s ds$. 
\end{lem}

\begin{proof}
The fact that $\mu_{\be^u}(T_{\be^u}\setminus L(T_{\be^u}))=0$ can be found in Aldous \cite{a3}. 
Now recall that by definition (see Lemma \ref{cont}), 
$\mu_{\be^u}=\lambda\circ \pi_{\be^u}^{-1}$, where $\lambda$ is the Lebesgue measure on $[0,u]$. 
Consequently, $\mu_{\be^u}(T_{\be^u})=u$.
Furthermore, 
$\int_{T_{\be^u}} d_{\be^u}(\rho,x) \mu_{\be^u}(dx)=\int_0^u d_{\be^u}(\rho,\pi_{\be^u}(s))ds$. 
But $\rho= \pi_{\be^u}(0)$,
whence, still by definition,
$d_{\be^u}(\rho,\pi_{\be^u}(s)) =\tilde d_{\be^u}(0,s)=\be^u_0+\be^u_s - 2 \min_{t\in[0,s]}\be^u_t=\be^u_s$.
\end{proof}

\subsection{Aldous's SSCRT}\label{ss3}

The random real tree we will be interested in, because it arises as a scaling
limit of the SBBCGWT, is, up to some scaling factor, 
the self-similar continuum random tree (SSCRT in short) of Aldous.
The definition we give is similar to
that of Aldous \cite[Section 6]{a1}, but the formalism is different.

\begin{defin}\label{mieux}
Consider $[0,\infty)$ endowed with its usual distance as a {\rm backbone}.
Let $N(dx,du,d\cT)$ be a Poisson measure
on $[0,\infty)\times (0,\infty)\times \tT$
with intensity measure $2 dx \bn(du) Q_u(d\cT)$, where  $\bn(du)= du / \sqrt{2\pi u^3}$.
For each mark $(x,u,\cT)$ of $N$, with $\cT=(T,d,\mu)$, plant $T$ on the backbone at $x$
(that is, connect the backbone and $\cT$ by identifying the root of $\cT$ and $x$).
At the end, we have a binary tree $H$, endowed with a distance $\delta$ (induced
by the distance on the backbone and the distances on the trees we have planted) and with a measure $\nu$ 
(induced by the measures on the trees we have planted). Say that the root 
$\rho$ is $0$ (of the backbone).
The random measured rooted real tree
$\cH=(H,\delta,\nu)$ is called a SSCRT. It holds that $\nu(H\setminus L(H))=0$.
\end{defin}

This last property is obviously inherited from Lemma \ref{crt}-(i).
The following lemma corresponds to another construction similar to Aldous \cite[Section 2.5]{a2}.

\begin{lem}\label{concon}
Consider a two-sided Bessel$(3)$-process $(C_t)_{t\in \rr}$, i.e.  $(C_t)_{t\geq 0}$
and $(C_{-t})_{t\geq 0}$ are two independent Bessel$(3)$-processes starting at $0$,
see Revuz-Yor \cite[Chapter VI]{ry}. The 
random measured rooted real tree $\cT_{C}$ built from $C$ as in Lemma 
\ref{cont2} is a SSCRT in the sense of Definition \ref{mieux}.
\end{lem}

All this is certainly contained in
Aldous \cite{a1,a2,a3} with a different formalism.

\begin{proof} 
We divide the proof into several steps.

\vip

{\it Step 1.} First, we use the excursion theory, see Revuz-Yor \cite[Chapter XII]{ry}, to build
a Bessel$(3)$-process.
For $u>0$, denote by $\cE_u$ the set of all continuous
functions from $[0,u]$ into $[0,\infty)$ and
by $R_u$ the law of the Brownian excursion with length $u$, which is a probability measure
on $\cE_u$. Set $\cE=\cup_{u > 0} \cE_u$.
Consider a Poisson measure $M(dx,du,d\be)$ on $[0,\infty)\times [0,\infty)\times \cE$ with intensity measure
$dx \bn(du) R_u(d\be)$. 
For $x\geq 0$, let $\tau_x=\int_0^x \int_0^\infty \int_\cE u M(dx,du,d\be)$.
The map $x \mapsto \tau_x$ is a.s. increasing on $[0,\infty)$.
Define, for $t\geq 0$, its inverse $L_t= \inf \{x \geq 0 : \tau_x >t\}$. 
The map $t\mapsto L_t$ is a.s. continuous
and nondecreasing. Observe that for all $x\geq 0$, $\tau_{x-}= \inf \{t\geq 0 : L_t=x\}$
and $\tau_{x}= \sup \{t\geq 0 : L_t=x\}$.
Consider $\cI=\{x\geq 0 : \tau_{x-} \ne \tau_{x}\}$.
For $x \in \cI$, we denote by
$(x,u(x),\be(x))$ the corresponding mark of $M$, with $\be(x)= (\be(x,t))_{t \in [0,u(x)]}$.
By construction, $u(x)=\tau_x-\tau_{x-}$ for all $x\in \cI$.
Define $\cA= \cup_{x \in \cI} (\tau_{x-},\tau_{x})$ and $\cB = [0,\infty)\setminus \cA$. 
Put $W_t= 0$ for all $t \in \cB$ and $W_t= \epsilon(x) \be(x,t-\tau_{x-})$ for
$t \in  (\tau_{x-},\tau_{x})$, the sign $\epsilon(x)$ being chosen with a fair coin toss for each $x$.
It holds that $W$ is a Brownian motion and that $L$ is its local time at $0$ and thus
$X=L+|W|$ is a Bessel$(3)$-process, see \cite[Chapter VI, Corollary 3.8]{ry}.
It holds that $L_t= \inf_{[t,\infty)} X_s$.

\vip

{\it Step 2.} Build two Bessel$(3)$-processes $X^1$ and $X^2$ as in Step 1, using two
independent Poisson measures $M^1$ and $M^2$.
Introduce the two-sided Bessel$(3)$-process $C$ defined by $C_t=X^1_t$ for $t\geq 0$
and $C_t=X^2_{-t}$ for $t\leq 0$. Build the tree $\cT_C$ as in Lemma \ref{cont2}.

\vip

{\it Step 3.} Introduce the Poisson measure $N$ on $[0,\infty)\times [0,\infty)\times \tT$, image of
$M^1+M^2$ by the map $(x,u,\be) \mapsto (x,u,\cT_\be)$.
The intensity measure of $N$ is $2 dx \bn(du) Q_u(d\cT)$, because the image measure of $R_u$ by
the map $\be \mapsto \cT_\be$ is $Q_u$, recall Definition
\ref{dfcrt}. Build the tree $\cH=(H,\delta,\nu)$
as in Definition \ref{mieux} with this Poisson measure $N$. 

\vip

{\it Step 4.} Finally, it is a tedious but straightforward exercise to show that 
$\cT_{C}$ and $\cH$ are a.s. isometric. Let us indicate that (i) the backbone $[0,\infty)$ of $H$
corresponds $\{\pi_C(s) : (s \geq 0$ and $L^1_s=X^1_s)$ or  $(s\leq 0$ and $L^2_{-s}=X^2_{-s})\}$ in $T_C$ 
(recall Lemma \ref{cont2})
(ii) for each mark $(x,u(x),\be(x))$ of $M^1$ the tree $\cT_{\be(x)}$ planted at $x$ on the 
backbone of $H$ corresponds to $\{\pi_C(s),s \in (\tau^1_{x-},\tau^1_x)\}$ in $T_C$ (observe that
$C_t=x+\be(x,t-\tau^1_{x-})$ for $t \in (\tau^1_{x-},\tau^1_x)$), (iii)
for each mark $(x,u(x),\be(x))$ of $M^2$ the tree $\cT_{\be(x)}$ planted at $x$ on the 
backbone of $H$ corresponds to $\{\pi_C(s),s \in (-\tau^2_x, -\tau^2_{x-})\}$) in $T_C$ (observe that
$C_t=x+\be(x,-\tau^2_{x-}-t)$ for $t \in (-\tau^2_x, -\tau^2_{x-}))$. We used the notation of Step 1
with the additional superscripts $1,2$.
\end{proof}

The following remark explains why the SSCRT should be useful to us.

\begin{rem}\label{sl} 
Let $\hG$ be a SBBCGWT as in Definition \ref{fgw}. For $p>0$, endow the edges of $\hG$ with a length
$p^{-1/2}$. This induces a distance $d_p$ on $\hG$. Consider the measure 
$\mu_p=4 p^{-1} \sum_{x\in \hG} \delta_x$.
Then we believe that $(\hG,d_p,\mu_p)$ tends to the SSCRT  $(H,\delta,\nu)$ as $p\to \infty$.
A similar fact is mentioned in Aldous \cite[Section 2.7]{a2}
and Duquesne \cite[Theorem 1.5]{du} proves a very general result, including at least
the fact that  $(\hG,d_p)$ tends to $(H,\delta)$.
\end{rem}

Let us handle some brief computations showing that the scales and constant factors are
correct. We fix $x>0$.

\vip

$\bullet$ Let $A^p_x \subset \hG$ be composed of all the vertices (including
roots and terminal points) of the 
Galton-Watson trees planted on $\{\star,\star 1, \dots, \star{\lfloor \sqrt p x  \rfloor}\}$,
see Figure \ref{fig3}. Then $\mu_p(A^p_x)=4|A^p_x|/p$ and one easily checks that 
$|A^p_x|=1+2\sum_1^{\lfloor \sqrt px \rfloor} |L(G(i))|$.
Recalling that $|L(G(1))|\sim (p_k)_{k\geq 1}$ and using the same arguments as
in the proof of Theorem \ref{mr2}, we see that 
$\E[e^{-q|L(G(1))|}]=1-\sqrt{1-e^{-q}}$, whence 
$\E[e^{-q\mu_p(A^p_x)}]= e^{-4q/p}(1-\sqrt{1-e^{-8q/p}})^{\lfloor \sqrt px \rfloor}$, which tends
to $\exp(-2x\sqrt{2q})$.

\vip

$\bullet$ Let $A_x \subset H$ be composed of the part of the 
backbone $[0,x]$ and of all the trees planted on $[0,x]$. Then
$\nu(A_x)=\int_0^x\int_0^\infty\int_{\tT}u N(dy,du,d\cT)$,
since the backbone has no mass and since 
for each mark $(y,u,\cT)$ of $N$, with $\cT=(T,d,\mu)$, it holds that
$\mu(T)=u$ (because $\cT$ is a CRT$(u)$ and due to Lemma \ref{crt}-(i)).
Consequently, 
$\E[e^{-q\nu(A_x)}]=\exp(-x \int_0^\infty (1-e^{-qu})\sqrt{2/(\pi u^3)}du)=\exp(-2x\sqrt{2q})$.
We used that, performing an integration by parts and the substitution $u=x^2$,
\begin{align}\label{reutil}
\int_0^\infty (1-e^{-qu})\sqrt{\frac 2 {\pi u^3}}du =
\int_0^\infty e^{-qu} \frac{2\sqrt{2} q du}{\sqrt {\pi u}}=
\int_0^\infty e^{-q x^2} \frac{4\sqrt 2 q dx}{\sqrt \pi}=
2\sqrt{2q}.
\end{align}

\subsection{The limit pruning procedure}\label{ss4}
We introduce the pruning procedure $\cP_\infty$.

\vip

\begin{defin}\label{defpi} 
Let $\cT=(T,d,\mu)$ be a measured rooted real tree. We define $\cP_\infty(\cT)
=(T_\infty,d_\infty,\mu_\infty)$ as follows.

\vip

{\it Step 1.} Consider a Poisson measure $M(dz,dx)$ 
on $\{(z,x): z \in T, x \in B_T(z)\}$
with intensity measure $(1/16)\mu(dz)\lambda_{B_T(z)}(dx)$, where $B_T(z)$ is 
the branch of $T$ joining
$z$ to the root $\rho$, endowed with the Lebesgue measure $\lambda_{B_T(z)}$.
For each $t\geq 0$, we define $T_t= \{x\in T: \delta(x,\rho)=t\}$
and we introduce the height $t_0=\inf\{t>0 : T_t=\emptyset\}\in [0,\infty]$
of $\cT$.
We activate all the marks of $M$.

\vip

{\it Step 2.} We explore the tree $T$ from the top until we arrive the the 
root $\rho$ (that is, we consider $T_t$, for $t$ decreasing from $t_0$ 
to $0$),
with the following rule: each time we encounter a point $X$ issued from an active 
mark $(Z,X)$ of $M$, we remove the tree above $X$ (i.e. the whole descendance 
of $X$), we replace it by a leave
and we deactivate all the marks $(Z',X')$ of $M$ with $Z'$ descendant of $X$ (which means that
$X \in B_T(Z')$).

\vip

{\it Step 3.} We end with a tree $T_\infty \subset T$, that we endow with 
the distance $d_\infty$ and the measure $\mu_\infty$, 
restrictions of $d$ and $\mu$ to $T_\infty$.
We set $\cP_\infty(\cT)=(T_\infty,d_\infty,\mu_\infty)$.

\vip

A mark of $M$ is said to be {\rm useful} if has generated a pruning at some
time in the procedure and {\rm useless} otherwise.
\end{defin}

The picture is very similar to Figure \ref{fig1}.
The pruning procedure $\cP_\infty$ is not clearly well-defined if the height of
$\cT$ is infinite.

\begin{prop}\label{piwd}
Let $\cH=(H,\delta,\nu)$ be a SSCRT. Then $\cP_\infty(\cH)$  
(the Poisson measure $M$ used by $\cP_\infty$ being picked conditionally on $\cH$)
perfectly makes sense: there is a.s. no need to start
the pruning procedure from the top of $H$.
\end{prop}

\begin{proof}
Let thus $\cH=(H,\delta,\nu)$ be a SSCRT. We call $B$ its backbone.
For $z\in H$, we define $\rho(z)= \arg\min \{\delta(x,z) : x\in B\}$. In other words,
$\rho(z)=z$ if $z\in B$ and $\rho(z)$ is the root of the planted tree to which $z$ belongs otherwise.

\vip

For each $a \in \nn$, consider the subset $[a,a+1]$ of $B$. Let $\Omega_a$ be the event on which
$M(\{(z,x): z \in H, x\in B_H(z), \rho(z)\in [a,a+1], \rho(x)\in [a,a+1]\})=2$, 
and,  denoting by $(Z_1,X_1)$ and $(Z_2,X_2)$ the corresponding marks, 
there holds $X_1\in B$, $X_2\in B$ and $a<X_1<\rho(Z_1)<X_2<\rho(Z_2)<a+1$
(which makes sense since $a,X_1,\rho(Z_1),X_2,\rho(Z_2),a+1$ all belong to $B$).

\vip

First, we claim that on $\Omega_a$, the mark $(Z_1,X_1)$ is a.s. useful, so that 
we can do as if the backbone was ending at $X_1$ and
ignore all the marks $(Z,X)$ of $M$ with $\rho(Z)>X_1$ (i.e. ignore all the marks sent by
trees planted at the right of $X_1$). Indeed, on $\Omega_a$, separate two cases.

\vip

$\bullet$ Either the mark $(Z_2,X_2)$ is useful, so that the backbone will be pruned at $X_2$, 
whence all the marks $(Z',X')$ with $Z'>X_2$ will be deactivated. Furthermore, we know from 
$\Omega_a$ that there are no marks $(Z',X')$ satisfying $X_1<X'<\rho(Z_1)<\rho(Z')<X_2$. Thus
$(Z_1,X_1)$ will indeed be useful.

\vip

$\bullet$ Or the mark at $(Z_2,X_2)$ is useless, so that it has necessarily been deactivated by
an useful mark $(Z_3,X_3)$ with $X_3\in (X_2,\rho(Z_2))$ 
(necessarily sent by a tree planted at the right of $a+1$), 
but then the backbone is pruned at $X_3$ and we conclude as previously that
the mark $(Z_1,X_1)$ will be useful.

\vip

Next, $(\Omega_a)_{a\in \nn}$ is clearly an independent family of events and it holds
that $p:=\Pr(\Omega_a)$ does not depend on $a$. We will check that 
$p>0$ and this will conclude the proof thanks to the Borel-Cantelli Lemma.

\vip

To prove that $p=\Pr(\Omega_0)>0$, we observe that $\Omega_0^1\cap \Omega_0^2\cap\Omega_0^3\subset
\Omega_0$, where

\vip

$\bullet$ $\Omega_0^1$ is the event that $M(A_1)=1$, where 
$A_1:=\{(z,x): z \in H, x\in B_H(z), \rho(z)\in [0,1/2], \rho(x)\in 
[0,1/2]\}$ and that the corresponding mark $(Z_1,X_1)$ satisfies $X_1\in B$;

\vip

$\bullet$ $\Omega_0^2$ is the event that $M(A_2)=1$, where 
$A_2:=\{(z,x): z \in H, x\in B_H(z), \rho(z)\in (1/2,1], \rho(x)\in 
(1/2,1]\}$ and that the corresponding mark $(Z_2,X_2)$ satisfies $X_2\in B$;

\vip

$\bullet$ $\Omega_0^3$ is the event that $M(A_3)=0$, where $A_3:=\{(z,x): z \in H, x\in B_H(z), 
\rho(z)\in (1/2,1], x\in [0,1/2]\}$ (here $[0,1/2]\subset B$).

\vip

Recall that conditionally on $\cH$,
$M$ is a Poisson measure on  $\{(z,x): z \in H, x \in B_H(z)\}$
with intensity measure $m(dz,dx):=(1/16)\nu(dz)\lambda_{B_H(z)}(dx)$. 
Hence, knowing $\cH$, the events
$\Omega_0^1$, $\Omega_0^2$ and  $\Omega_0^3$ are 
independent (because the sets $A_1,A_2,A_3$ are a.s. pairwise disjoint).
Knowing $\cH$, $M(A_1)$ is Poisson$(m(A_1))$-distributed and it 
obviously holds that $m(A_1)\in (0,\infty)$ a.s. (because
$\nu$, restricted to the set of the trees planted on the subset $[0,1/2]$
of the backbone is a.s. a finite and positive measure
and because a.s., $\sup_{z\in H, \rho(z)\in[0,1/2]} |B_H(z)|<\infty$). Hence, $\Pr(M(A_1)=1 \vert \cH)=m(A_1)\exp
(-m(A_1))>0$ a.s. But knowing $\cH$, that $M(A_1)=1$ and knowing $Z_1$ (here we denote by $(Z_1,X_1)$
the mark of $M$ in $A_1$), $X_1$ belongs to $[0,\rho(Z_1))\subset B$ with positive
probability (since $X_1$ is uniformly distributed on the branch $B_{H}(Z_1)$). All this proves that
$\Pr(\Omega_0^1 \vert \cH)>0$ a.s. Similar arguments show that 
$\Pr(\Omega^2_0 \vert \cH)>0$ and $\Pr(\Omega^3_0 \vert \cH)>0$ a.s.,
whence $\Pr(\Omega_0 \vert \cH)>0$ a.s. and thus $p=\Pr(\Omega_0)>0$.
\end{proof}

\subsection{Heuristic scaling limit}\label{ss5}

Consider the pruned SBBCGWT $\hG_n=\cP_n(\hG)$ as in Proposition \ref{treesOK2}, endow its edges with the 
length $n^{-1/3}$, which
induces a distance $\delta_n$ on $\hG_n$. Also introduce the measure
$\nu_n=4 n^{-2/3}\sum_{x\in \hG_n} \delta_x$ on $\hG_n$. Our aim in this paragraph 
is to explain heuristically
why we expect that $(\hG_n,\delta_n,\nu_n)$ converges, as $n \to \infty$, to the 
pruned SSCRT $\cP_\infty(\cH)=(H_\infty,d_\infty,\nu_\infty)$. 

\vip

Consider $p$ very large, endow each edge of $\hG$  with the length $p^{-1/2}$ (which induces
a distance $d_p$ on $\hG$) and set 
$\mu_p=4 p^{-1}\sum_{x \in \hG} \delta_x$. Then we know from
Remark \ref{sl} that $(\hG,d_p,\mu_p)$ resembles the SSCRT $\cH=(H,\delta,\nu)$.
Next consider $n$ very large (with $n >> p^{1/2}$) and the 
pruning procedures $\cP_n$ applied to $\hG$.
Then for each internal node $z\in N(\hG)$, the Poisson process $\pi^n_z$
sends 

\vip

$\bullet$ 0 mark with approximate probability $1-|B_\hG(z)|p^{1/2}/(2n)$, where $|B_\hG(z)|=d_p(\rho,z)$ 
is the length of the branch $B_\hG(z)$ with $\hG$ endowed with $d_p$. Indeed,
in the original scales, we have a Poisson process with rate $1/n$ on the branch $B_\hG(z)$ 
which is composed of $p^{1/2}|B_\hG(z)|$ edges with $\expo(2)$-distributed length and the 
expectation of $\expo(2)$ is $1/2$;

\vip

$\bullet$ one mark with approximate probability $|B_\hG(z)|p^{1/2}/(2n) $ 
(and then this mark is uniformly distributed in $B_\hG(z)$);

\vip

$\bullet$ two or more marks with very small probability.

\vip

Hence we can say roughly that the internal nodes $z\in N(\hG)$ that do send one mark at $x \in B_{\hG}(z)$ are 
chosen according to a Poisson measure on $\{(z,x): z \in N(\hG), 
x \in B_\hG(z)\}$ 
with intensity $m_{p,n}(dx,dz):=\frac {p^{1/2}}{2n} \sum_{y\in N(\hG)} \delta_y(dz)\lambda_{B_\hG(z)}(dx)$.
Since each leave is very close to an internal node and since there are (roughly) as much
leaves as internal nodes in a binary tree, we expect that 
$$
m_{p,n}(dx,dz) \simeq \frac{p^{1/2}}{4n} \sum_{y\in \hG} \delta_y(dz)\lambda_{B_\hG(z)}(dx)
= \frac{p^{3/2}}{16n} \mu_p(dz)\lambda_{B_\hG(z)}(dx).
$$
Thus choosing $p=n^{2/3}$, whence $p^{3/2}/(16n)=1/16$, 
$(\hG,d_{n^{2/3}},\mu_{n^{2/3}})$ resembles 
the SSCRT $\cH=(H,\delta,\nu)$ and the pruning procedure $\cP_n$ of
$(\hG,d_{n^{2/3}},\mu_{n^{2/3}})$
resembles the pruning procedure $\cP_\infty$ of $\cH$.
Calling finally $\delta_n$ the restriction of $d_p$ to $\hG_n$ (which consists in endowing each
edge of $\hG_n$ with a length $p^{1/2}=n^{1/3}$) and $\nu_n$ the restriction of $\mu_p$
to $\hG_n$ (which gives a weight $4/p=4n^{-2/3}$ to each vertex of $\hG_n$), we deduce
that $(\hG_n,\delta_n,\nu_n)$ resembles $\cH_\infty$.

\subsection{The measure of leaves}\label{ss6}

The main result of this section is the following.

\begin{theo}\label{measleaves}
Consider a SSCRT $\cH=(H,\delta,\nu)$ as in Lemma \ref{mieux} and build
$\cP_\infty(\cH)=(H_\infty,\delta_\infty,\nu_\infty)$ as in Definition \ref{defpi} (picking the Poisson measure $M$
used by $\cP_\infty$ conditionally on $\cH$).
Then the law of $\nu_\infty(H_\infty)/8$ is $xc(x)dx$, where $c$ was defined in Theorem
\ref{mr3}.
\end{theo}

Adopt the notation of Paragraph \ref{ss5}. Write 
$n^{-2/3}|L(\hG_n)| \simeq n^{-2/3}|\hG_n|/2 = \nu_n(\hG_n)/8$,
which should tend to $\nu_\infty(H_\infty)/8$.
Since $|L(\hG_n)| \sim (kc^n_k)_{k\geq 1}$ by Proposition \ref{treesOK2} and since
$\nu_\infty(H_\infty)/8 \sim xc(x)dx$ by Theorem \ref{measleaves}, 
all this is coherent with Theorem \ref{mr3}.

\vip

We start with two technical lemmas, of which the proof lies at the end of the
section.

\begin{lem}\label{tl1}
Recall (\ref{psiex}). There holds
$$
\int_0^\infty \left(1 - e^{|a_1'|z}\psiex(\sqrt 2 z) \right) \frac{dz}{2\sqrt \pi z^{3/2}}=0.
$$
\end{lem}

\begin{lem}\label{equfonc}
Let $p:[0,\infty)\mapsto (0,1]$ be continuous, decreasing, satisfy $p(0)=1$ and
$\lim_\infty p=0$. Assume that for all $q>0$,
\begin{align}\label{ef}
1-p(q)=\int_0^\infty p(s) h(q,s)ds,
\end{align}   
where, the function $\ell$ being defined by (\ref{lala}),
$$
h(q,s)= |a_1'|[\ell((q+s)/2)- \ell(s/2)]\frac{\Ai^2((q+s)/2+a_1')}{\Ai^2(q/2+a_1')}.
$$
Then for all $q\in[0,\infty)$, $p(q)=\ell'(q/2)$.
\end{lem}

Next we compute, in some sense, the mass of the pruned planted trees.

\begin{lem}\label{nel}
For each $u>0$, denote by $m_u$ the law of the mass of the pruned CRT$(u)$.
In other words, pick $\cT_u=(T_u,d_u,\mu_u)$ a CRT$(u)$, prune it as in Definition \ref{defpi},
set $\cP_\infty(\cT_u)=(T_{u,\infty},d_{u,\infty},\mu_{u,\infty})$: $m_u$ is the law of $\mu_{u,\infty}(T_{u,\infty})$.
Define the measure $\bnel$ on $(0,\infty)$
by
$$
\bnel(A)=\int_0^\infty m_u(A) \bn(du),
$$
where $\bn(du)=du/\sqrt{2\pi u^3}$ as in Definition \ref{mieux}. It holds that
$$
\bnel(du)= e^{|a_1'|u/8} \psi_{ex}(u^{3/2}/16) \bn(du)
$$
where $\psi_{ex}$ was defined in (\ref{psiex}).
\end{lem}

As will be clear from the proof, we are not able to compute the law $m_u$
for a given value of $u$. We really need to choose $u$ according to the measure $\bn(du)$.

\begin{proof} We introduce a Poisson measure $N(dx,du,d\cT)$ on $[0,1]\times(0,\infty)\times \tT$ 
with intensity measure $dx \bn(du) Q_u(d\cT)$, we denote by $\{(x,u_x,\cT_x)\}_{x \in \cI}$ the set of its marks,
where $\cT_x=(T_x,d_x,\mu_x)$ with root $\rho_x$. Recall that for each $x\in \cI$,
$\mu_x(T_x)=u_x$ by Lemma \ref{crt}. 
Conditionally on $N$, we prune $\cT_x$ for each $x\in \cI$ independently, using $\cP_\infty$ with 
a Poisson measure $M_x$ on $\{(z,y) : z \in T_x, y \in B_{T_x}(z)\}$ with intensity measure 
$(1/16)\mu_x(dz)\lambda_{B_{T_x}(z)}(dy)$
and denote by $\cT_{x,\infty}=(T_{x,\infty},d_{x,\infty},\mu_{x,\infty})$ the resulting tree.
Then we consider 
\begin{align*}
Z:=&\sum_{x\in\cI}u_x=\int_0^1 \int_0^\infty\int_\tT u N(dx,du,d\cT),\\
Z_\infty:=&\sum_{x\in\cI} \mu_{x,\infty}(T_{x,\infty}).
\end{align*}
\vip

{\it Step 1.} Here we prove that for any $r\in (-\infty,\infty)$,
$$
\E\left(e^{-r Z_\infty} \right)= \exp\left(-\int_0^\infty \left(1-e^{-ru}\right)\bnel(du) \right) \in [0,\infty].
$$
Conditionally on $\{(x,u_x), x \in \cI \}$, the random variables $\mu_{x,\infty}(T_{x,\infty})$ 
(for $x\in \cI$) are independent and  $m_{u_x}$-distributed, so that
\begin{align*}
&\E\left(e^{-r Z_\infty} | \{(x,u_x)\}_{x \in \cI} \right)= \prod_{x\in\cI} \int_0^\infty e^{-r v} m_{u_x}(dv)\\
=&\exp\left( \int_0^1\int_0^\infty\int_\tT \left[\log \int_0^\infty e^{-rv}m_u(dv)\right]N(dx,du,d\cT)  \right).
\end{align*}
Taking expectations, we get
\begin{align*}
\E\left(e^{-r Z_\infty}\right)=&\exp\left(\int_0^1 dx \int_0^\infty \bn(du) \int_\tT Q_u(d\cT) 
\left(e^{\log \int_0^\infty e^{-rv}m_u(dv)} -1\right) \right)\\
=&\exp\left(\int_0^\infty \bn(du)\left(\int_0^\infty e^{-rv}m_u(dv) -1 \right) \right)\\
=&\exp\left(\int_0^\infty \bn(du) \int_0^\infty m_u(dv) \left(e^{-rv} -1\right) \right).
\end{align*}
We have finished, since $\bnel(dv)=\int_0^\infty \bn(du)m_u(dv)$ by definition.

\vip

{\it Step 2.} The aim of this step is to prove that 
$$
\Pr\left(\left. Z=Z_\infty  \right| \{(x,u_x)\}_{x \in \cI} \right) = \exp\left(    
\int_0^1\int_0^\infty\int_\tT \log \psiex( u^{3/2} / 16 ) N(dx,du,d\cT)
\right).
$$
We have $Z=Z_\infty$ iff $\cT_{x,\infty}=\cT_x$ for all $x\in\cI$. Conditionally
on $N$, the events $\{\cT_{x,\infty}=\cT_x\}_{x \in \cI}$ are independent. Furthermore,
$\{\cT_{x,\infty}=\cT_x\}=\{M_x$ has no mark$\}$, which occurs with probability
$\exp(- \int_{T_x} \mu_x(dz)  d_x(z,\rho_x)/16)$, because for each $z\in T_x$, the length 
of $B_{T_x}(z)$ is $d_x(z,\rho_x)$.
But knowing $u_x$, $\cT_x$ is a CRT$(u_x)$, whence 
by Lemma \ref{crt},  $\int_{T_x} d_x(z,\rho_x)\mu_x(dz)$ can be written as
$(u_x)^{3/2}\int_0^1 \be_s ds$, for some normalized Brownian excursion. We deduce that 
$\Pr(\cT_{x,\infty}=\cT_x | u_x)=\psiex(u_x^{3/2}/16)$. Finally, by independence,
$\Pr(Z=Z_\infty | \{(x,u_x)\}_{x \in \cI})= \prod_{x\in\cI}\psiex((u_x)^{3/2}/16)$, which ends the step.

\vip

{\it Step 3.} Next we prove that there is $\lambda>0$ such that
$$
\Pr\left(\left. Z=Z_\infty  \right| Z_\infty \right) = \exp\left(- \lambda Z_\infty \right).
$$
This is the most delicate part of the proof.

\vip

{\it Step 3a: Branching property for the forest of CRTs.} Here we recall some results found in 
Duquesne-Le Gall \cite[Section 4.2]{dlg2} (in the more general case of L\'evy trees).
First, we denote by $T=\cup_{x\in \cI} T_x$ the forest of CRTs, endowed with the distance $d$ and measure
$\mu$ induced by $(d_x,\mu_x)_{x\in \cI}$, and by $\cT$ the corresponding measured forest 
(with $d(y,y')=\infty$ if $y\in T_x$
and $y'\in T_{x'}$ for some $x\ne x'$). For $y\in T$, define the height of $y$ as
$h(y)=d(y,\rho_x)$, where $\rho_x$ is the root of the tree $T_x$ to which $y$ belongs.
For $a\geq 0$, we introduce $T(a)=\{y \in T : h(y)=a\}$, by $T([0,a])=\{y \in T : h(y)\leq a\}$, 
by $\cT([0,a])$ the corresponding measured forest, and we define similarly $T([a,\infty))$
and $\cT([a,\infty))$.

For $a\geq 0$, we define the point measure $\cN_a$ on $T(a)\times \tT$ as
$\cN_a:= \sum_{i\in \cJ_a} \delta_{(x_i^a,\cS_i^a)}$, where $(\cS_i^a)_{i \in \cJ_a}\in\tT$ are the closures of the 
connected components of $T \setminus T([0,a])$ endowed with the inherited distances and measures 
and $x_i^a$ is the root of $\cS_i^a$.

We also introduce the $\sigma$-finite measure $\bQ$ on $\tT$
defined by $\bQ(.)=\int_0^\infty \bn(du)Q_u(.)$. 

Then, by \cite{dlg2}, there a.s. exists, for each $a\geq 0$, a finite measure $k^a$ on $T$ 
such that the following points hold:

(i) $k^0=0$ and for all $a>0$, supp $k^a\subset T(a)$,

(ii) for all $a>0$, conditionally on $\cT([0,a])$, the point measure $\cN_a$ is Poisson with intensity
$k^a(dx)\bQ(dS)$,

(iii) $\mu(.)=\int_0^\infty k^a(.)da$.

\vip

{\it Step 3b: Branching property for the marked forest of pruned CRTs.} First, we call
$\widehat \tT$ the set of marked measured real trees: its elements are of the form
$(\cS,A)$ for $\cS=(S,d_S,\mu_S)\in\tT$ and for $A$ a finite subset of $L(S)$.
Define the pruned forest $\cT_\infty=(T_\infty,d_\infty,\mu_\infty)$ with $T_\infty=\cup_{x\in \cI}T_{x,\infty}$
similarly as in Step 3a. 
When pruning $\cT$ to build $\cT_\infty$, we keep track of the final useful marks
and denote by $A_\infty\subset L(T_\infty)$ the leaves of $\cT_\infty$ where a pruning has occurred (i.e.
$A_\infty = L(T_\infty) \setminus L(T)$). We put $\widehat \cT_\infty=(\cT_\infty,A_\infty)$,
which is a forest of elements of $\widehat \tT$.
We also define $T_\infty(a)$, $T_\infty([0,a])$, $\cT_\infty([0,a])$,  
$T_\infty([a,\infty))$, $\cT_\infty([a,\infty))$ for $a\geq 0$ as in Step 3a.
We denote by $\widehat \cT_\infty([0,a])=(\cT_\infty([0,a]),A_\infty\cap T_\infty([0,a]))$ and define
similarly $\widehat \cT_\infty([a,\infty))$.

For $a\geq 0$, we define the measure  $\widehat \cN_{a,\infty}
:= \sum_{i\in \cJ_{a,\infty}} \delta_{(x_{i,\infty}^a,\widehat \cS_{i,\infty}^a)}$ on $T_\infty(a)\times \widehat \tT$,
where
$(\widehat \cS_{i,\infty}^a)_{i \in \cJ_a}$ are the (marked and measured) closures of the 
connected components of $T_\infty \setminus T_\infty([0,a])$ together with $(\{y\},0,0,\{y\})$ (the tree
$\{y\}$ endowed with the distance $0$, the measure $0$ and with one mark at $y$) 
for $y\in A_\infty \cap T_\infty(a)$
(which is a.s. of cardinality $0$ or $1$), and 
where $x_{i,\infty}^a$ is the root of
$\widehat \cS_{i,\infty}^a$.

We finally introduce the measure
$\widehat \bQ_\infty(.)=\int_0^\infty \bn(du) \widehat Q_{u,\infty}(.)$ on $\widehat \tT$, where 
$\widehat Q_{u,\infty}$
is the law of the marked pruned CRT$(u)$ (that is, consider a CRT$(u)$, prune
it, keep track of the useful marks, this gives a marked measured tree, with law $\widehat Q_{u,\infty}$).

The aim of this step is to check that, denoting, for each $a\geq 0$, by $k^a_\infty$
the restriction of $k^a$ (which is a measure on $T$) to $T_\infty$ (which is a subset of $T$),
the following points hold:

(i) $k^0_\infty=0$ and for all $a>0$, supp $k^a_\infty \subset T_\infty(a)$,

(ii) for all $a>0$, conditionally on $\widehat \cT_\infty([0,a])$, the point measure 
$\widehat \cN_{a,\infty}$ is Poisson with intensity $k^a_\infty(dx)e^{-a \mu_S(S)/16}\widehat\bQ_\infty(d \widehat\cS)$
(with $\widehat \cS=(S,d_S,\mu_s,A_S)\in \widehat \tT$),

(iii) $\mu_\infty(.)=\int_0^\infty k^a_\infty(.)da$.

\vip
First, points (i) and (iii) are directly inherited from from Step 3a.
To prove point (ii), we fix the height $a$. 
We are going to decompose the pruning procedure according 
to the decomposition of the tree $T$ (at height $a$) introduced at Step 3a. 
Consider the Poisson measure $M=\sum_{x\in\cI}M_x$ on 
$\{ (z,y) : z \in T, y \in B_{T}(z)\}$ (recall from the first paragraph of the proof that the measure 
$M_x$ is used to prune $T_x$, so that for $x\ne y$, $M_x$ and $M_y$ have disjoint supports).

Given  $\cT$, we write $M=M^a + \sum_{i \in \cJ_a}( \overline{M}^{a}_i + \underline{M}^{a}_i)$ where 
$M^a$ is supported by $\{(z,y) : z \in T([0,a]), y \in B_T(z)\}$ while, for all $i \in \cJ_a$,  
$\overline{M}^{a}_i$  is supported by $\{ (z,y) : z \in S_i^a, y \in B_{T}(z), h(y) \geq a \}$ 
and $\underline{M}^{a}_i$  is supported  by 
$\{ (z,y) : z \in S_i^a, y \in B_{T}(z), h(y)< a \}=\{ (z,y) : z \in S_i^a, y \in B_{T}(x_i^a)\}$.  

We start the pruning procedure of $T$ by $M$ (from the top of $T$). Clearly,
we first prune each subtree $\cS_i^a$ (for $i$ in $\cJ_a$) using only $\overline{M}_i^a$. 
Conditionally on $\cT([0,a])$, the marked pruned tree $\widehat \cS^a_{i,\infty}$ we obtain is
measurable with respect to the $\sigma$-fields generated by $\cS^a_i$ and $\overline{M}_i^a$. 
It now follows from the branching property (Step 3a, point (ii)) and from  the definition of 
$\widehat{\bQ}_\infty$ that, conditionally on $\cT([0,a])$, the point measure  
$\sum_{i\in \cJ_a} \delta_{(x_i^a,\widehat{\cS}_{i,\infty}^a)}$ is Poisson with intensity 
$k^a(dx) \widehat\bQ_\infty(d \widehat\cS)$. 

For all $i \in \cJ_a$, we consider the event on which all the marks sent by
the leaves of $S_i^a$ ``under $a$'' have been deactivated by the pruning procedure ``above $a$'', i.e.
$$
\Omega^a_i = \{\underline{M}_i^a (\{ (z,y) : z \in S_{i,\infty}^a, 
y \in B_{T}(z), h(y)<a\}) = 0\}.
$$
This event is measurable with respect to (the $\sigma$-field generated by) $\cS^a_{i}$, 
$\underline{M}^a_{i}$ and $\overline{M}^a_{i}$. 
For all $i \in \cJ_a$, we have $P(\Omega^a_i \, | \, \cT([0,a]), \cS^a_{i,\infty}) 
= e^{-a \mu_\infty(S^a_{i,\infty})/16}$ (here we used that $a \mu(S^a_{i,\infty})/16$ 
is the integral of the intensity of $\underline{M}_i^a$).
It follows that, conditionally on $ \cT([0,a])$, the point measure 
$\widehat{\cN}_{a,\infty}^0:= \sum_{i\in \cJ_a} \indiq_{\Omega^a_i}\delta_{(x_i^a,\widehat{\cS}_{i,\infty}^a)}$
is Poisson with intensity $k^a(dx)  e^{-a \mu_S(S)/16} \widehat{\bQ}_\infty(d\widehat{\cS})$.

On the one hand,  $\widehat{\cT}_\infty([0,a])$ is measurable with respect to  $\cT([0,a])$ and to the
(useful) marks of $M$ falling under $a$ (i.e. the useful marks $(z,y)$ with $h(y)<a$). 
On the other hand, for all $i \in \cJ_a$, the event
$\Omega_i^a$ guarantees us that $\underline{M}_i ^a$ has no (useful) mark under height $a$. Thus 
$\widehat{\cT}_\infty([0,a])$ 
is, conditionally on $\cT([0,a])$, 
independent on $\widehat{\cN}_{a,\infty}^0$. Consequently, 
the restriction $\widehat{\cN}_{a,\infty}^1:=\sum_{i\in \cJ_a} \indiq_{\Omega^a_i}\indiq_{\{x_i^a \in {\rm supp} \;k^a_\infty\}}
\delta_{(x_i^a,\widehat{\cS}_{i,\infty}^a)}$ of $\widehat{\cN}_{a,\infty}^0$ to the 
support of $k^a_{\infty}$ is, conditionally on $\widehat{\cT}_\infty([0,a])$, Poisson with intensity 
$k^a_\infty(dx)  e^{-a \mu_S(S)/16} \widehat{\bQ}_\infty(d\widehat{\cS})$. 
 
We further observe that if $i \in \cJ_{a,\infty}$, then $\Omega^a_i$ is fulfilled: if $S_{i,\infty}^a$
was sending a mark below $a$, then either this mark would be useful and remove 
$S_{i,\infty}^a$ or this mark would not be useful and would be removed, together with 
$S_{i,\infty}^a$, by another useful mark.
This shows that $\widehat{\cN}_{a,\infty}$ is nothing but $\widehat{\cN}_{a,\infty}^1$.
Point (ii) follows. 
 
\vip

{\it Step 3c: Independence.} Here we prove that for any $a\geq 0$, $A_\infty \cap T_\infty([0,a])$
is independent of $A_\infty \cap T_\infty([a,\infty))$ conditionally on $\cT_\infty$.
This means that conditionally on the (not marked) pruned tree $\cT_\infty$, the marks above $a$
are independent of the marks under $a$.

We fix $a\geq 0$ and we consider the point measure $\cN^2_{a,\infty}
:= \sum_{i\in \cJ_{a,\infty}} \delta_{(x_{i,\infty}^a,\cS_{i,\infty}^a, \cS_{i}^a)}$ on $T_\infty(a)\times \tT \times \tT$. 
The same argument as in Step 3b shows that, conditionally on  $\widehat{\cT}_\infty([0,a])$, it is a
Poisson measure with  intensity  $k^a_\infty(dx)  \bR(d\cS,d\cT)e^{-a \mu_{T}(T)}$, where 
$\bR(.)=\int_0^\infty \bn(du) R_u(.)$, for $R_u$  the joint law of $(\cS_\infty,\cS)$, where $\cS$
is a CRT$(u)$ and $\cS_\infty$ is the resulting pruned tree.

Conditionally on $\widehat{\cT}_\infty([0,a])$ and  $\cT_\infty$,  $(x_{i,\infty}^a,\cS_{i,\infty}^a)_{i\in\cJ_{a,\infty}}$
is completely determined. For each 
$i \in \cJ_{a,\infty}$, the conditional law of $\cS_i^a$ knowing $\widehat{\cT}_\infty([0,a])$ and $\cT_\infty$
depends only on $\cS_{i,\infty}^a$ and in particular does not depend on $A_\infty \cap T_\infty([0,a])$.
Since finally $A_\infty \cap T_\infty([a, \infty])= \cup_{i \in \cJ_{a,\infty}} L(\cS_{i,\infty}^a)
\setminus L(\cS_{i}^a)$, the claim follows.

\vip

{\it Step 3d: Homogeneity.} The aim of this step is to check that for $0\leq a<b$, conditionally
on $\cT_\infty$ and knowing that $\#(A_\infty \cap(T_\infty(a)\cup T_\infty(b)))=1$, 
we have $\#(A_\infty\cap T_\infty(a))=1$ with probability
$k^{a}_\infty(T_\infty)/[k^{a}_\infty(T_\infty)+k^{b}_\infty(T_\infty)]$.

We consider two independent copies $\widehat{\cT}_\infty$ and $\widehat{\cT}'_\infty$. We are going 
to apply the branching property at height $a$ for the first one and at height $b$ for the second one. 
We define the point measure 
$$ 
\widehat \cN_{a,b,\infty} := \sum_{i\in \cJ_{a,\infty}} \delta_{(x_{i,\infty}^a,\widehat \cS_{i,\infty}^a)} 
+ \sum_{i\in \cJ^{'}_{b,\infty}} \delta_{({x}_{i,\infty}^{'b},\widehat {\cS}_{i,\infty}^{'b})}
$$ 
on 
$(T_\infty(a) \cup T'_\infty(b)) \times \widehat \tT$. Independence 
(between $\widehat{\cT}_\infty$ and $\widehat{\cT}'_\infty$) and the branching 
property (Step 3b, point (ii)) shows that, conditionally on $\widehat{\cT}_\infty([0,a])$ 
and $\widehat{\cT}'_\infty([0,b])$, the measure $\widehat \cN_{a,b,\infty}$ is Poisson with intensity 
measure
$$
(k^a_\infty (dx) e^{-a\mu_S(S)} + k^{'b}_\infty (dx) e^{-b\mu_S(S)}  ) \widehat{\bQ}_\infty (d\widehat{\cS}).
$$

Consider the trivial marked tree $\widehat{\tau}= (\{0\},0,0,\{0\}) \in \widehat{\tT}$
consisting of one (marked) point (the root).
Standard properties of Poisson 
measures show that knowing $\widehat{\cN}_{a,b,\infty} \left( (T_\infty(a)\cup T'_\infty(b)) \times 
\{ \widehat{\tau}\} \right)=1$ and denoting by $(X,\hat \tau)$ the corresponding mark of 
$\widehat{\cN}_{a,b,\infty}$, it holds that conditionally on $\widehat{\cT}_\infty([0,a])$ 
and $\widehat{\cT}'_\infty([0,b])$,
$X \in T_\infty(a)$ with probability
$k^a_\infty(T_\infty)/[ k^a_\infty (T_\infty) + {k'}^b_\infty(T'_\infty)]$.
(Here we use that for $\widehat{\cS}=\widehat{\tau}$, we have $e^{-a\mu_S(S)}=e^{-b\mu_S(S)}=1$). 
But $(X,\hat \tau)$ is a mark of $\widehat{\cN}_{a,b,\infty}$ if and only if 
$X \in A_\infty\cap T_\infty(a)$ or $X \in A'_\infty\cap T'_\infty(b)$. We have shown that
\begin{align*}
&\Pr[\#(A_\infty\cap T_\infty(a))=1 \vert \#((A_\infty\cap T_\infty(a))\cup(A'_\infty\cap T'_\infty(b) ))=1,
\cT_\infty]\\
=& \frac{k^a_\infty(T_\infty)}{k^a_\infty (T_\infty) + {k'}^b_\infty(T'_\infty)}.
\end{align*}

To conclude, we fix $h\in(a,b)$. As a consequence of Step 3c, given $\cT_\infty$ and 
$\{\cT_\infty = \cT_\infty'\}$,  the marked tree 
$\widehat{\cT}_\infty([0,h]) \cup \widehat{\cT}_\infty'([h,\infty])$ has the same distribution as 
$\widehat{\cT}_\infty= \widehat{\cT}_\infty([0,h]) \cup \widehat{\cT}_\infty([h,\infty])$. 
Hence, the conditional distribution  given  $\cT_\infty$ and $\{\cT_\infty = \cT_\infty'\}$
and $\{\#((A_\infty \cap T_\infty(a))\cup(A'_\infty \cap T'_\infty(b)))=1\} $ of the unique element of 
$(A_\infty \cap T_\infty(a)) \cup (A'_\infty \cap T'_\infty(b))$ is the same as the  conditional distribution  
given  $\cT_\infty$ and $\{\#((A_\infty \cap T_\infty(a))\cup((A_\infty \cap T_\infty(b)))=1\} $ of the unique 
element of $A_\infty \cap (T_\infty(a) \cup  T_\infty(b))$.
The claim follows. 

\vip

{\it Step 3e: Conclusion.} Using Steps 3c and 3d, one could prove that conditionally on 
$\cT_\infty$, the set $A_\infty$ is a Poisson point process on $L(T_\infty)$ with intensity proportional
to $\mu_\infty$. We will prove a slightly weaker result, which suffices to conclude.

Let us denote by $R_\infty(a)=\#\{A_\infty \cap T_\infty([0,a])\}$, which counts the number of marks below
the height $a$ of $T_\infty$. This process is clearly c\`adl\`ag, nondecreasing, and has only jumps of height
$1$ (it is clear from the pruning procedure that there will a.s. never be two marks at the same height).
Step 3c implies that conditionally on $\cT_\infty$, for all $0<a<b$, $R_\infty(b)-R_\infty(a)$ is
independent of $\sigma(\{R_\infty(h) : h \in [0,a]\})$. Hence $(R_\infty(a))_{a\geq 0}$ is a Poisson
process with some intensity measure $m_\infty(da)$ (this random measure being $\sigma(\cT_\infty)$-measurable).
Next, Step 3d tells us that for any $0\leq a < b$, $\Pr[\Delta R_\infty(a)=1 \vert \Delta R_\infty(a)+
\Delta R_\infty(b)=1, \cT_\infty]= k^{a}_\infty(T_\infty)/[k^{a}_\infty(T_\infty)+k^{b}_\infty(T_\infty)]$. We classically
deduce that $m_\infty(da)$ is of the form $\lambda k^a_\infty(T_\infty)da$, for some $\lambda >0$.
Consequently,
$$
\Pr(R_\infty(\infty)=0 \vert \cT_\infty)= \exp\left(- \lambda \int_0^\infty k^a_\infty(T_\infty)da  \right)
= \exp\left(- \lambda \mu_\infty(T_\infty)  \right),
$$
where we finally used Step 3b-(iii). Observe now that by definition, $Z_\infty=\mu_\infty(T_\infty)$,
that $\{Z=Z_\infty\}= \{\cT=\cT_\infty\}=\{A_\infty=\emptyset\}=\{R_\infty(\infty)=0\}$ and
that $Z_\infty$ is $\sigma(\cT_\infty)$-measurable. We conclude that
$$
\Pr(Z=Z_\infty \vert Z_\infty)= \E[\Pr(R_\infty(\infty)=0 \vert \cT_\infty )\vert Z_\infty]
= \exp\left(- \lambda Z_\infty  \right) 
$$
as desired.
\vip

{\it Step 4.} Writing $\E[e^{-rZ}\indiq_{\{Z=Z_\infty\}}]\in[0,\infty]$ in two different ways using Steps 2 and 3
and that $Z$ is $\sigma(\{(x,u_x), x \in \cI \})$-measurable, 
we deduce that for any 
$r\in (-\infty,\infty)$, $I(r)=J(r)$, where
\begin{align*}
I(r) :=& \E \left[\exp\left(-rZ+ \int_0^1\int_0^\infty\int_\tT \log \psiex( u^{3/2} / 16 ) N(dx,du,d\cT)\right) 
\right],\\
J(r) := & \E \left[\exp\left(-rZ_\infty- \lambda Z_\infty\right)\right].
\end{align*}
Using that
\begin{align*}
I(r) =& \E \left[\exp\left(-\int_0^1\int_0^\infty\int_\tT [ru-\log \psiex( u^{3/2} / 16 )] N(dx,du,d\cT)\right) 
\right]\\
= & \exp\left(-\int_0^\infty \left(1-e^{-ru+\log \psiex( u^{3/2} / 16 )}\right)\bn(du) \right)
\end{align*}
as well as Step 1, we deduce that for all $r$, 
$$
\int_0^\infty \left(1-e^{-(r+\lambda)u}\right)\bnel(du)=
\int_0^\infty \left(1-e^{-ru}\psiex( u^{3/2} / 16 )\right)\bn(du)
$$
and thus that for all $q$,
$$
\int_0^\infty \left(1-e^{-qu}\right)\bnel(du)=
\int_0^\infty \left(1-e^{-(q-\lambda)u}\psiex( u^{3/2} / 16 )\right)\bn(du).
$$
Choosing $q=0$, we see that $\int_0^\infty \left(1-e^{\lambda u}\psiex( u^{3/2} / 16 )\right)\bn(du)=0$.
This can clearly be satisfied only for one value of $\lambda$, and we infer
from Lemma \ref{tl1} that $8\lambda = |a_1'|$ (use the substitution $u=8z$). We have
shown that for any $q\geq 0$,
$$
\int_0^\infty \left(1-e^{-qu}\right)\bnel(du)=
\int_0^\infty \left(1-e^{-qu}e^{|a_1'|u/8}\psiex( u^{3/2} / 16 )\right)\bn(du).
$$
Differentiating this equality with respect to $q$ and using the injectivity of the Laplace
transform, we conclude that $u\bnel(du) = u e^{|a_1'|u/8}\psiex( u^{3/2} / 16 )\bn(du)$
as desired.
\end{proof}

The following observation is somewhat obvious in view of Lemma \ref{nel}.

\begin{lem}\label{buissonsdab}
Let $B=[0,\infty)$ stand for the backbone.
Let $R(dx,du,dy)$ be a Poisson measure on $B\times (0,\infty) \times (0,\infty)$ with intensity
measure $r(dx,du,dy):=2 dx \bnel (du) (u/16)e^{-uy/16}dy$, where $\bnel(du)$ was defined and computed in
Lemma \ref{nel}.
We activate all the marks of $R$ and 
explore $B$ from $\infty$ to $0$ with the following rule: each time we encounter a point $T=X-Y$,
where $(X,U,Y)$ is an active mark of $R$, we prune $B$ at $T$ and we deactivate all the marks
$(X',U',Y')$ of $R$ such that $X'>T$. We end with a finite segment $[0,\tau]\subset B$. 
Then 
$$
\Gamma:=\int_B \int_0^\infty\int_0^\infty u \indiq_{\{x<\tau\}}R(dx,du,dy)
$$
has the same law as $\nu_\infty(H_\infty)$, for some pruned  SSCRT $\cP_\infty(\cH)
=(H_\infty,\delta_\infty,\nu_\infty)$.
\end{lem}

\begin{proof} 
The main idea is just that we can prune first the planted trees and then the backbone.
Consider a SSCRT $\cH=(H,\delta,\nu)$ as in Definition \ref{mieux}, built
with a Poisson measure $N(dx,du,d\cT)$ with intensity measure $2 dx \bn(du)Q_u(d\cT)$, of which
we denote by $\{(x,u_x,\cT_x), x\in \cI\}$ the marks. 
Conditionally on $N$, prune $\cT_x$ independently 
for each $x\in \cI$, and denote by  $\cT_{x,\infty}=(T_{x,\infty},d_{x,\infty},\mu_{x,\infty})$ the resulting tree.
Consider the point measure $O=\sum_{x \in \cI} \delta_{(x,\mu_{x,\infty}(T_{x,\infty}))}$ on
$B\times (0,\infty)$. Due to Lemma \ref{nel}, $O$ is a Poisson measure with intensity
$2 dx \bnel(du)$. 

\vip

Then we prune the backbone as follows. For each $x\in\cI$, we consider a Poisson process $\pi_x$ 
on $[0,x) \subset B$ with rate $\mu^x_\infty(T^x_\infty)/16$. We activate all these Poisson processes.
Then we explore $B$ from $\infty$ to $0$ with the following rule: each time we encounter a mark
$T$ of an active Poisson process $\pi^x$, we prune the backbone at $T$ and we deactivate
all the Poisson processes $\pi^y$ with $y>T$. We end with a finite segment $[0,\tau]\subset B$.
Then $[0,\tau]$, on which are planted the (pruned) trees $\cT_{x,\infty}$ for $x \in [0,\tau]\cap \cI$,
is a clearly pruned SSCRT $\cP_\infty(\cH)=(H_\infty,\delta_\infty,\nu_\infty)$ and its total mass is nothing but 
$\nu_\infty(H_\infty)=\sum_{x \in \cI} \mu_{x,\infty}(T_{x,\infty})\indiq_{\{x < \tau\}}
=\int_B \int_0^\infty u \indiq_{\{x<\tau\}}O(dx,du)$.

\vip

Finally, we observe that for each $x\in \cI$, at most one mark of $\pi^x$
can be useful, and this mark is of the form $x-Y_x$, where $Y_x$ is exponentially distributed with
parameter $\mu_{x,\infty}(T_{x,\infty})/16$ (actually, $\pi^x$ has one (or more) mark if  $x-Y_x>0$ and 
no mark if $x-Y_x<0$). Finally, we find the pruning procedure of the statement, with the 
Poisson measure $R=\sum_{x\in\cI} \delta_{(x,\mu_{x,\infty}(T_{x,\infty}),Y_x)}$ which has indeed the intensity $r$
and for which $\nu_\infty(H_\infty)=\sum_{x \in \cI} \mu_{x,\infty}(T_{x,\infty})\indiq_{\{x < \tau\}}
=\int_B \int_0^\infty\int_0^\infty u \indiq_{\{x<\tau\}}R(dx,du,dy)$.
\end{proof}

We can finally give the

\begin{preuve} {\it of Theorem \ref{measleaves}.}
We consider the random variable $\Gamma$ built in Lemma \ref{buissonsdab}.
We have to show that $\Gamma/8 \sim xc(x)dx$. We introduce $p(q):=\E[e^{-q\Gamma/16}]$ for $q\geq 0$.
Recall from Theorem \ref{mr3} that
\begin{align*}
\ell(q)= \int_0^\infty (1-e^{-q x}) c(x)dx \quad \hbox{whence} \quad 
\ell'(q)= \int_0^\infty e^{-q x} x c(x)dx.
\end{align*}
By injectivity of the Laplace transform, we thus only have to check that for all $q\geq 0$,
$p(2q)=\ell'(q)$. To this end, we will use Lemma \ref{equfonc}. Since $p$
is clearly continuous, decreasing, satisfies $p(0)=1$ and $p(\infty)=0$ (because $\Gamma>0$ a.s.),
it suffices to prove that $p$ satisfies (\ref{ef}). We thus fix $q>0$ for the whole proof.

\vip

{\it Step 1.} Recall the construction handled in Lemma \ref{buissonsdab}. Denote by
$\{(x,u_x,y_x) : x \in \cI\}$ the marks of $R$.
We consider the event
$\Omega(q)$ on which all the marks $(X,U,Y)$ of $R$ such that $X-Y\in [-q,0]$ are finally deactivated
(at the end of the procedure).

\vip

{\it Step 2.} We prove that $\Pr(\Omega(q))=p(q)$. By construction, 
all the marks $(X,U,Y)$ of $R$ such that $X>\tau$ are deactivated (because $\tau$ prunes
the backbone) and all the marks $(X,U,Y)$ of $R$ such that $X<\tau$ are still active
and are known to satisfy $Y>X$ (because else, the backbone would have been pruned at $X-Y<\tau$).
Thus $\Pr[\Omega(q) \vert \{x,u_x\}_{x\in \cI},\tau]= 
\prod_{x\in \cI \cap [0,\tau]} \Pr[y_x>x+q \vert y_x> x,u_x]=\prod_{x\in \cI \cap [0,\tau)}e^{-qu_x/16}$,
because $y_x$ is $\cE xp(u_x/16)$-distributed (conditionally on $u_x$). We immediately conclude,
since $\sum_{x\in \cI \cap [0,\tau)} u_x= \Gamma$, that $\Pr[\Omega(q) \vert \{x,u_x\}_{x\in \cI},\tau]
=e^{-q \Gamma/16}$, whence $\Pr(\Omega(q))=\E[e^{-q \Gamma/16}]$ as desired.

\vip

{\it Step 3.} Next, we introduce
\begin{align*}
V :=& \inf \left\{x>0 : R(\{x\}\times (0,\infty)\times (x,x+q))>0\right\}, \\
W:=& \inf \left\{x>0 : R(\{x\}\times (0,\infty)\times (0,x))>0\right\},
\end{align*}
and we denote by $(V,U_V,Y_V)$ and $(W,U_W,Y_W)$ the corresponding marks.
We also consider the event $A$ on which all the marks $(X,U,Y)$ of $R$ such that $X>V$ and 
$X-Y \in [0,V]$ are finally
deactivated when applying the pruning procedure described in Lemma \ref{buissonsdab} only on
$(V,\infty)$ (i.e. exploring $B$ only from $\infty$ to $V+$). It holds that 
$$
(\Omega(q))^c = \{V<W\}\cap A.
$$
Indeed,
\vip

$\bullet$ On $\{V<W\}\cap A$, all the marks $(X,U,Y)$ with $X>V$ and $X-Y\in [0,V]$ 
are deactivated (thanks to $A$) and
there are no mark $(X,U,Y)$ with $0<X-Y<X<V$ (because $W>V$), so that the mark $(V,U_V,Y_V)$
will not be deactivated and it holds that $V-Y_V \in [-q,0]$ by construction, whence
$\Omega(q)$ cannot be realized.

\vip

$\bullet$ On $A^c$, there is a useful mark $(X,U,Y)$ such that $0<X-Y <V<X$, which thus makes
$(V,U_V,Y_V)$ inactive (as well as all the marks $(X',U',Y')$ with $X'>V$). Since by definition
of $V$, there is no active mark $(X,U,Y)$ with $X<V$ and $X-Y \in[-q,0]$, we deduce that $\Omega(q)$
is realized.

\vip

$\bullet$ On $\{V>W\}$, we have to consider two situations. Either the mark $(W,U_W,Y_W)$ is useful,
whence it prunes the backbone at $W-Y_W\in(0,V)$ and thus makes $(V,U_V,Y_V)$ inactive 
(and it also deactivates all the marks $(X',U',Y')$ with $X'>V$) and we conclude as
previously that $\Omega(q)$
is realized. Or the mark $(W,U_W,Y_W)$ is useless, but then it has necessarily been deactivated by
a useful mark $(X,U,Y)$ such that $0<X-Y<W$ which also deactivates $(V,U_V,Y_V)$,
since $V>W$. We conclude again as previously.

\vip

{\it Step 4.} By an argument of invariance by translation, we see that 
$\Pr[A \vert V]=p(V)$, recalling that $\Pr[\Omega_0(q)]=p(q)$ by Step 2.
Indeed, conditionally on $V=v$, $A$ is the same event as $\Omega_0(v)$ concerning the Poisson
measure $R^v=\sum_{x\in \cI\cap(v,\infty)}\delta_{(x-v,u_x,y_x)}$.
Next, the events $A$ and $\{V<W\}$ are independent conditionally on $V$.
Indeed, conditionally on $V=v$, the event $A$ concerns only the marks $\{(x,u_x,y_x),x\in\cI\cap(v,\infty)\}$
while $\{W<v\}$ concerns only the marks $\{(x,u_x,y_x),x\in\cI\cap(0,v)\}$.
Using Steps 2 and 3, we deduce that 
$$
1-p(q)=\E(p(V)\Pr[V<W\vert V]).
$$

{\it Step 5.} Define $\psi(q)=\int_0^\infty (1-e^{-qu/16})\bnel(du)$. Here we show that 
\begin{align}\label{presque}
\E(p(V)\Pr[V<W\vert V])= 2\int_0^\infty p(v) \left[\psi(v+q)-\psi(v)\right]e^{-2\int_q^{q+v} \psi(x)dx} dv.
\end{align}
First observe that for $v>0$, $\{V>v\}=\{R(I_v)=0\}$ and $\{W>w\}=\{R(J_w)=0\}$, where
$I_v=\{(x,u,y): x \in [0,v], u>0, y \in [x,x+q]\}$ and 
$J_w=\{(x,u,y): x \in [0,w], u>0, y \in [0,x]\}$. Since $I_v \cap J_w = \emptyset$ for any values
of $v>0,w>0$, we deduce that $V$ and $W$ are independent. 
Recall that the intensity measure of $R$ is $r(dx,du,dy):=2 dx \bnel (du) (u/16)e^{-uy/16}dy$. 
First, $\Pr(V>v)=e^{-r(I_v)}$, with
\begin{align*}
r(I_v)=&2 \int_0^v dx \int_0^\infty \bnel(du) \int_x^{x+q} dy \frac{u}{16}e^{-u y/16}\\
=& 2 \int_0^v dx \int_0^\infty \bnel(du) [e^{-u x/16}-1+1  -e^{-u (x+q)/16} ]\\
=&  2 \int_0^v dx \left[\psi(x+q)-\psi(x)\right].
\end{align*}
Consequently, the density of the law of $V$ is 
$$
h_V(v)=2\left[\psi(v+q)-\psi(v)\right] e^{- 2\int_q^{q+v} \psi(x)dx+2\int_0^v \psi(x)dx}.
$$
Next, we have $\Pr(W>w)=e^{-r(J_w)}$ with
\begin{align*}
r(J_w)=&2 \int_0^w dx \int_0^\infty \bnel(du) \int_0^x dy \frac{u}{16}e^{-u y/16}=  2 \int_0^w dx  \psi(x).
\end{align*}
By independence, we deduce that $\Pr(W>V\vert V)=\exp\left(- 2 \int_0^V dx  \psi(x) \right)$.
Formula (\ref{presque}) immediately follows.

\vip

{\it Step 6.} We finally deduce that $p$ satisfies (\ref{ef}).
First, we observe from Lemma \ref{nel} and Theorem \ref{mr3} that
$$
\bnel(du) = \frac{|a_1'|}{16} c(u/8)du.
$$
This implies, recalling that $\ell(q)=\int_0^\infty (1-e^{-qu})c(u)du$, that
$\psi(q)=|a_1'|\ell(q/2)/2$.  We thus find, using Steps 4 and 5, that
$$
1-p(q)=\int_0^\infty p(v) |a_1'| \left[ \ell((v+q)/2) - \ell(v/2)\right] e^{-|a_1'|\int_q^{q+v} \ell(x/2)dx} dv.
$$
To conclude, it only remains to observe from (\ref{lala}) that
\begin{align*}
e^{-|a_1'|\int_q^{q+v} \ell(x/2)dx}=&\exp\left(\int_q^{q+v} \frac{\Ai'(x/2+a_1')}{\Ai(x/2+a_1')}dx\right)
= \frac{\Ai^2((q+v)/2+a_1')}{\Ai^2(q/2+a_1')}.
\end{align*}
This ends the proof.
\end{preuve}

It remains to prove the technical lemmas. We start with the

\begin{preuve} {\it of Lemma \ref{tl1}.}
Let us call $I$ the integral in the LHS. Write $I=K+L$, where
\begin{align*}
K=\int_0^\infty  \left(1 - e^{|a_1'|z}\right)\psi_{ex}(\sqrt 2z)\frac{dz}{2\sqrt \pi z^{3/2}},\quad
L=\int_0^\infty \left(1 - \psi_{ex}(\sqrt 2 z) \right) \frac{dz}{2\sqrt \pi z^{3/2}}.
\end{align*}
First, using (\ref{for}) (that can be extended to all values of $q\in (a_1,\infty)$
by analyticity) with $q=a_1'$ and that $\Ai'(a_1')=0$, we get $K=\Ai'(0)/\Ai(0)$. 
Next, we write $L=\lim_{q\to\infty} 
(L_q^1-L^2_q)$,
where 
\begin{align*}
L^1_q=\int_0^\infty  \left(1 - e^{-qz}\right)\frac{dz}{2\sqrt \pi z^{3/2}},\quad
L^2_q=\int_0^\infty \left(1 - e^{-qz}\right)\psi_{ex}(\sqrt 2 z) \frac{dz}{2\sqrt \pi z^{3/2}}.
\end{align*}
Using (\ref{reutil}), we get $L^1_q=\sqrt q$ and (\ref{for}) implies
$L^2_q= \Ai'(0)/\Ai(0) - \Ai'(q)/\Ai(q)$.
Consequently, $I= \lim_{q\to\infty} (\sqrt q +\Ai'(q)/\Ai(q))$. But 
$\Ai'(q)/\Ai(q) = - \sqrt q + O(1/q)$ as $q\to \infty$, see e.g. Janson \cite[Section 16]{j},
so that $I=0$ as desired.
\end{preuve}

We finally handle the

\begin{preuve} {\it of Lemma \ref{equfonc}.}
We first prove that there is at most one function $p$ satisfying the assumptions
of the statement. To do so, we first observe that, since $\ell$ is increasing by definition, see
(\ref{lala}), $h(q,s)>0$ as
soon as $q>0$. Next, a simple computation shows that for all $q\in(0,\infty)$,
\begin{align*}
\kappa(q):=& \int_0^\infty |h(q,s)|ds \\
<& \int_0^\infty |a_1'|\ell((q+s)/2)\frac{\Ai^2((q+s)/2+a_1')}{\Ai^2(q/2+a_1')} ds\\
=& \frac{-1}{\Ai^2(q/2+a_1')} \int_0^\infty \Ai'((q+s)/2+a_1')\Ai((q+s)/2+a_1') ds =1.
\end{align*}
Consider now two functions $p_1,p_2$ satisfying the assumptions of the statement, set $u=p_1-p_2$. It
holds that $||u||_\infty=u(q_0)$ for some $q_0\in (0,\infty)$, because
$u(0)=0$, $\lim_\infty u=0$ and because $u$ is continuous.
Observe next that
$$
||u||_\infty=|u(q_0)|=\left|\int_0^\infty u(s)h(q_0,s)ds \right| \leq \kappa(q_0)||u||_\infty,
$$
whence $||u||_\infty=0$ as desired since $\kappa(q_0)<1$.

\vip

We now prove that $p(q):=\ell'(q/2)$ satisfies all the requirements. First, by Theorem \ref{mr3}, there
holds $p(q)=\int_0^\infty e^{-qx/2}xc(x)dx$, where $xc(x)dx$ is a probability measure on $(0,\infty)$, so
that $p$ is clearly decreasing, continuous and satisfies  $p(0)=1$ and $\lim_\infty p=0$.
To prove that (\ref{ef}), it suffices
to check that $H(q)=I(q)+J(q)$, where $H(q):=\Ai^2(q/2+a_1')(1-\ell'(q/2))$, where
\begin{align*}
I(q):=& |a_1'|\int_0^\infty \ell'(s/2) \ell((q+s)/2)\Ai^2((q+s)/2+a_1') ds \\
=& - \int_0^\infty \ell'(s/2) \Ai((q+s)/2+a_1')\Ai'((q+s)/2+a_1') ds,
\end{align*}
(here we used the explicit expression (\ref{lala}) of $\ell$) and where
\begin{align*}
J(q):=& -|a_1'|\int_0^\infty \ell'(s/2) \ell(s/2)\Ai^2((q+s)/2+a_1') ds \\
=& |a_1'| \int_0^\infty \ell^2(s/2) \Ai((q+s)/2+a_1')\Ai'((q+s)/2+a_1') ds
\end{align*}
(here we used an integration by parts and that $\ell(0)=0$). Recalling (\ref{equadiff}), we deduce that
$|a_1'| \ell^2(s/2)-\ell'(s/2) = s/(2|a_1'|)-1$, whence $I(q)+J(q)=K(q)+L(q)$, where
\begin{align*}
K(q):=& - \int_0^\infty \Ai((q+s)/2+a_1')\Ai'((q+s)/2+a_1') ds=\Ai^2(q/2+a_1')
\end{align*}
and where
\begin{align*}
L(q):=& \frac 1 {2|a_1'|} \int_0^\infty s \Ai((q+s)/2+a_1')\Ai'((q+s)/2+a_1') ds \\ 
=& -\frac 1 {2|a_1'|} \int_0^\infty \Ai^2((q+s)/2+a_1') ds \\
=& - \frac 1{2|a_1'|}\int_q^\infty \Ai^2(s/2+a_1') ds.
\end{align*}
Gathering everything, we realize that it only remains to check that $L(q)=M(q)$, where (here we use
(\ref{equadiff}) and then the explicit expression (\ref{lala}) of $\ell$)
\begin{align*}
M(q):=&-\ell'(q/2)\Ai^2(q/2+a_1')\\
=&-[1-q/(2|a_1'|)+|a_1'|\ell^2(q/2)]\Ai^2(q/2+a_1')\\
=&-\Ai^2(q/2+a_1')+q\Ai^2(q/2+a_1')/(2|a_1'|)-(\Ai'(q/2+a_1'))^2/|a_1'|.
\end{align*}
Since $\lim_\infty M=\lim_\infty L=0$, it suffices to check that $M'(q)=L'(q)$, 
i.e. that $2|a_1'| M'(q)= \Ai^2(q/2+a_1')$. But, recalling that $\Ai''(x)=x\Ai(x)$,
\begin{align*}
2|a_1'| M'(q) :=&(-2|a_1'|+q)\Ai(q/2+a_1')\Ai'(q/2+a_1')+\Ai^2(q/2+a_1')\\
&-2\Ai'(q/2+a_1')\Ai''(q/2+a_1')\\
=&(q+2a_1')\Ai(q/2+a_1')\Ai'(q/2+a_1')+\Ai^2(q/2+a_1')\\
&-2(q/2+a_1')\Ai(q/2+a_1')\Ai'(q/2+a_1'),
\end{align*}
which equals $\Ai^2(q/2+a_1')$ as desired.
\end{preuve}

\subsection{The contour process point of view}\label{ss7}

We now give an alternative presentation of the pruned SSCRT. It might seem simpler 
at first glance, but when transposing the whole section to the formalism below, 
the proofs become rather less transparent.

\begin{rem}\label{cpov}
Consider a two-sided Bessel$(3)$-process $(C_t)_{t\in \rr}$ as in Lemma \ref{concon}.
Consider a Poisson measure $O(dt,dy)$ on the half plane $\rr\times [0,\infty)$ 
with intensity measure $(1/16)dsdy$.
\vip

Start with $(D_t)_{t\in \rr}=(C_t)_{t\in \rr}$ and with $\Pi$ the set of all the marks $(T,Y)$ 
of $O$ that fall under the graph of $(C_t)_{t\in \rr}$ (i.e. for which $Y < C_T$).

\vip

Travel from $y=\infty$ to $y=0$ (that is, consider the horizontal lines with ordinate $y$ for
$y$ going from $\infty$ to $0$) with the following rule:
each time we encounter a remaining mark $(T,Y)\in \Pi$, we update $(D_t)_{t\in (L,R)}$ and $\Pi$ 
as follows.

\vip

(i) Consider the largest (for inclusion) subset $\tcl s,t \tcr\subset \rr$ containing $T$ (this notation was
introduced in Lemma \ref{cont2}) such that $\tcl s,t \tcr \times \{Y\}$ is under the graph of $(D_t)_{t\in \rr}$
(i.e. such that for all $u \in \tcl s,t \tcr$, $C_u\geq Y$).

\vip

(ii) Remove the zone $\tcl s,t \tcr \times [0,\infty)$ from the picture (in particular, 
this makes disappear all the marks $(T',Y')\in \Pi$ with $T'\in \tcl s,t\tcr$) and glue if necessary:

(a) if $\tcl s,t \tcr=[s,t]$ for some $0\leq s \leq t$, remove everything in the zone 
and then translate the picture on the right of $t$ by $t-s$ to the left;

(b)  if $\tcl s,t \tcr=[s,t]$ for some $s \leq t \leq 0$, remove everything in the zone 
and then translate the picture on the left of $s$ by $t-s$ to the right;

(c)  if $\tcl s,t \tcr=(-\infty,s]\cup[t,\infty)$ for some $s<0< t$, remove everything in the zone.

\vip

We end with a compactly supported function $(D_t)_{t \in [L,R]}$, with $L<0<R$, satisfying 
$D_L=D_R$. We finally build the continuous function $(C^\infty_t)_{t\in [0,R-L]}$ by setting
$C^\infty_t=D_t$ for $t \in [0,R]$ and $C^\infty_t=D_{L-R+t}$ for $t\in[R,R-L]$. It holds that
$C^\infty_0=C^\infty_{R-L}=0$.
Build the tree $\cT_{C^\infty}$ as in Lemma \ref{cont}. Then $\cT_{C^\infty}$ is a pruned SSCRT
$\cP_\infty(\cH)=(H_\infty,\delta_\infty,\mu_\infty)$. In particular, 
$\Gamma:=R-L$ equals $\nu_\infty(H_\infty)$, so that $\Gamma/8$ is $xc(x)dx$-distributed.
See Figure \ref{fig4} for an illustration.
\end{rem}

\begin{figure}[hb]
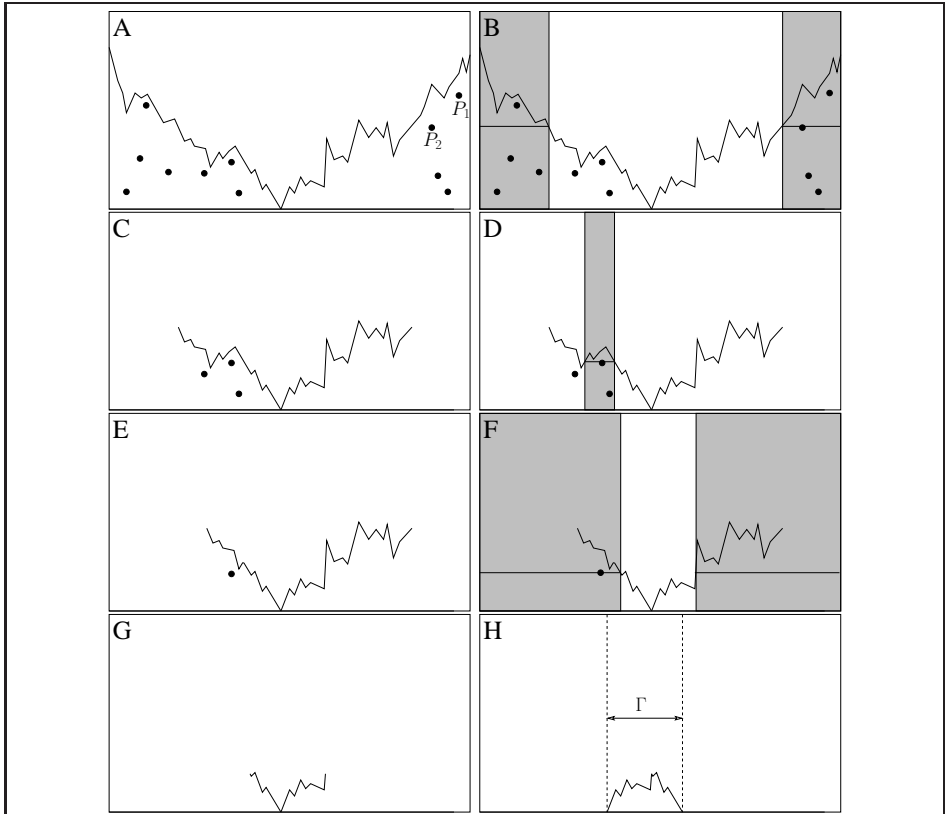

\fbox{
\begin{minipage}[c]{0.95\textwidth}
\centering
\resizebox{0.4\linewidth}{!}{\input{PruneCont1.pstex_t}}
\resizebox{0.4\linewidth}{!}{\input{PruneCont2.pstex_t}}\\ 
\resizebox{0.4\linewidth}{!}{\input{PruneCont3.pstex_t}}
\resizebox{0.4\linewidth}{!}{\input{PruneCont4.pstex_t}}\\ 
\resizebox{0.4\linewidth}{!}{\input{PruneCont5.pstex_t}} 
\resizebox{0.4\linewidth}{!}{\input{PruneCont6.pstex_t}}\\ 
\resizebox{0.4\linewidth}{!}{\input{PruneCont7.pstex_t}} 
\resizebox{0.4\linewidth}{!}{\input{PruneCont8.pstex_t}} 
\caption{Pruning the SSCRT: the contour point of view}
\label{fig4}
\vip
\parbox{12cm}{
\footnotesize{
On Figure A, the two-sided Bessel$(3)$-process $(C_t)_{t\in \rr}$ is drawn, as well as the marks 
of the Poisson measure $O(dt,dy)$ with intensity $(1/16)dtdy$ (we have erased all the marks above
the graph of $C$).
We assume that for $t$ outside the picture (both on the left and right), 
$C_t$ never goes under $C_{T_1}$, where $P_1=(T_1,Y_1)$. This allows
us to claim that the mark $P_2$ is useful (this is obviously the case if $P_1$ is useful,
and this is also the case if $P_1$ is useless, because then it has been erased, as well as all the marks
outside the picture.
Thus $P_2$ is useful. On Figure B, we thus build the corresponding 
zone, which we erase in Figure C. In Figure D, we 
consider the highest remaining mark and the corresponding zone,
which we erase in Figure E. The last mark is treated in Figures F and G.
Since there are no more marks, we have finished. In Figure H, we exchange the left and 
right parts of remaining function.
Build the tree with this contour function as in Lemma \ref{cont}. What we get is 
a pruned SSCRT $\cP_\infty(\cH)=(H_\infty,\delta_\infty,\mu_\infty)$. In particular, $\Gamma=\nu_\infty(H_\infty)$, 
so that $\Gamma/8$ is $xc(x)dx$-distributed.
}}
\end{minipage}
}
\end{figure}

Let us explain briefly why this builds a pruned SSCRT.
Recall from Lemma \ref{concon} that the SSCRT $\cH=(H,\delta,\nu)$ can be built 
from $(C_t)_{t\in\rr}$ as in Lemma \ref{cont2}: consider
the pseudo-distance $\tilde d_{C}$ on $\rr$, the equivalence relation $\sim^C$
on $\rr$, the quotient space $T_{C}=\rr/\sim^C$, the canonical
projection $\pi_{C}:\rr\mapsto T_{C}$, the distance $d_{C}$ induced by $\tilde d_{C}$ 
on $T_{C}$, the measure $\mu_{C}= \lambda \circ \pi_{C}^{-1}$ on $T_{C}$ (where $\lambda$ is the Lebesgue
measure on $\rr$), the root $\rho=\pi_{C}(0)$. Define
$J_t= \inf_{[t,\infty)}C_u$ if $t\geq 0$ and $J_t= \inf_{(-\infty,t]}C_u$ if $t\leq 0$.
Set $B=\{\pi_{C}(t) : t \in \rr, C_t=J_t\}$. Finally,
set $\cH=(H,\delta,\nu)$ with $H=T_C$, $\delta=d_C$, $\nu=\mu_C$.

\vip

We introduce $A=\{(t,y) : t \in \rr, 0<y<C_t\}$.
For $t>0$ and $y\in(0,C_t)$, we introduce $s=\sup\{u\in(0,t] : C_u=y\}$ and set $\gamma_t(y)=\pi_C(s)$.
For $t<0$ and $y\in(0,C_t)$, we consider $s=\inf\{u\in [t,0) : C_u=y\}$ and set $\gamma_t(y)=\pi_C(s)$.
Observe that for all $(t,y) \in A$, $\gamma_t(y)\in H$
is the point of the branch $B_H(\pi_C(t))$ with $\delta(\rho,\pi_C(s))=y$.

\vip

Next, we build a Poisson measure $M(dz,dx)$ on $\{(z,x) : z \in H, x \in B_H(z)\}$ with intensity
measure $(1/16)\nu(dz) \lambda_{B_H(z)}(dx)$, where $B_H(z)$ is the branch joining $\rho$
to $z$, endowed with the Lebesgue measure. 
Denote by $\{(T_i,Y_i) \}_{i \geq 1}$ the marks of $O$ in $A$.
For each $i\geq 1$, let $Z_i=\pi_{C}(T_i)$ and $X_i=\gamma_{T_i}(Y_i)$.
Then $M:=\sum_{i\geq 1} \delta_{(Z_i,X_i)}$ is a Poisson measure on 
$\{(z,x) : z \in H, x \in B_H(z)\}$ with the correct intensity measure.
Indeed, $M$ is the image measure of $O$ (restricted to $A$) by the map
$f(t,x)=(\pi_C(t),\gamma_t(x))$. Hence conditionally on $\cH$, $M$ is a Poisson measure on
$\{(z,x) : z \in H, x \in B_H(z)\}$ with intensity measure $(1/16) \lambda_A \circ f^{-1}$,
$\lambda_A$ being the Lebesgue measure on $A$. Using that $\nu=\lambda_{(-\infty,\infty)}\circ \pi_C^{-1}$
by definition, one easily concludes that $M$ is as desired.

\vip

Consider a mark $(T,Y)$ of $O$, and the associated mark $(Z,X)=(\pi_C(T),\gamma_T(Y))$ of $M$.
Assume that at some time in the procedure, the mark produces a pruning.

\vip

(a) From the trees point of view (see $\cP_\infty$), we remove the whole descendance of $X$
and deactivate (forever) the marks $(Z',X')$ of $M$ with $Z'$ descendant of $X$.

\vip

(b) From the contour processes point of view (see Remark \ref{cpov}), we consider the largest 
subset $\tcl s,t \tcr \subset \rr$ such that  $\tcl s,t \tcr\times \{Y\} \subset A$ and
remove the whole zone $\tcl s,t \tcr\times (0,\infty)$ from the picture.

\vip

One easily understands that (b) is exactly the same thing as (a), because
the subtree of $H$ consisting of all the descendance of $X$ is nothing but 
$\{\pi_C(u) : u \in \tcl s,t \tcr\}$ and the marks $(Z',X')$ of $M$ with $Z'$ descendant of $X$
correspond to the marks $(T',Y')$ of $O$ such that $T'\in \tcl s,t \tcr$.

\vip

At the end, we get a function $(D_t)_{t \in [L,R]}$. Roughly, $(D_t)_{t\in[0,R]}$ is the ``upper contour''
of our pruned SSCRT, and $(D_t)_{t\in[L,0]}$ is the ``lower contour'' of our pruned SSCRT.
We need to identify the backbones. Exchanging the left and right parts of $(D_t)_{t \in [L,R]}$
to get the function $(C^\infty_t)_{t \in [0,R-L]}$ and using  $(C^\infty_t)_{t \in [0,R-L]}$ as a contour
function is a way to do it.

\subsection{A noticeable diffusion process}\label{ss8}

We conclude the paper with a remark about a diffusion process related
to pruned CRTs. This is related to Lemma \ref{nel}. Consider a Brownian motion $(B_t)_{t\geq 0}$.
Then exactly as in Lemma \ref{concon}, 
$(|B_t|)_{t\geq 0}$ can be seen as the contour function of a family of CRT$(u)$, 
planted at $x\in[0,\infty)$ with a Poisson measure $M(dx,du,d\cT)$ with intensity
measure $dx (2\pi u^3)^{-1/2}du Q_u(d\cT)$ (here there is no backbone: the CRTs are disconnected).
Prune each of these CRTs independently using $\cP_\infty$. From the contour point
of view (as in Subsection \ref{ss7}), we can handle the following procedure: consider a
Poisson measure $O(dt,dx)$ on $\rr\times[0,\infty)$ with intensity measure
$dtdx/16$. Treat independently each excursion of $(|B_t|)_{t\geq 0}$: start from the top of the excursion,
etc...
We obtain a new continuous nonnegative path  $(Y_t)_{t\geq 0}$, which is the contour function
of an infinite forest.

\vip

We believe that we can write $Y_t=|X_t|$, where $X_t$ solves the equation $X_t=dW_t - |a_1'|\ell(X_t/2)dt /2.$
The main idea is that when at height $x$, the process $X_t$ is as a Brownian motion, 
but the pruning procedure pushes it down, and we found the drift coefficient by handling
formal computations. We do not want to prove such a result, because the paper
is long enough, but we believe that the following proposition is quite noticeable.


\begin{prop}
Fix $\alpha>0$ and consider the odd function $\beta:\rr\mapsto \rr$
defined by 
$$
\beta(x)=- {\rm sign}(x) \alpha \frac{\Ai'(\alpha |x| + a_1')}
{\Ai(\alpha |x|+ a_1')}.
$$ 
Let $(B_t)_{t\geq 0}$ be a standard Brownian motion, let 
$(X_t)_{t\geq 0}$ be the unique solution to
$$
X_t= B_t - \int_0^t \beta(X_s)ds
$$
and let $(\tau_x)_{x\geq 0}$
be the inverse of the local time $(L_t)_{t\geq 0}$ at $0$ of $(X_t)_{t\geq 0}$. 
Then it holds that 
$\E[e^{-\lambda \tau_x}]=\exp(-x \psi(\lambda))$, where
\begin{align}\label{edla}
\psi(\lambda)=\beta(2\lambda /\alpha^3)
=\int_0^\infty (1-e^{-\lambda u}) \frac{\exp(\alpha^2|a_1'|u/2)}
{\sqrt{2\pi u^3}} \psiex(\alpha^3u^{3/2}/2)du,
\end{align}
where $\psiex$ was defined in (\ref{psiex}).
In particular if $\alpha=2^{1/3}$ then we have the noticeable property 
that the drift coefficient $\beta$ coincides with the Laplace exponent $\psi$ 
of the inverse local time.
\end{prop}

What is surprising in this result is that the drift coefficient and the Laplace exponent
of the inverse local time are two objects of completely different nature.

\begin{proof}
First observe that $\beta$ is $C^1$ on $\rr$ and has at most linear growth
(because $-\Ai'(x)/\Ai(x) \sim \sqrt x$ as $\to +\infty$, see e.g. Janson \cite[Section 16]{j}), 
so that the S.D.E. has a unique solution. Next, the second equality in (\ref{edla}) 
is immediately deduced from (\ref{lala}).

\vip

Using the result of Pitman-Yor \cite[Theorem 1]{py},
the Laplace exponent $\psi$ of the inverse local time $(\tau_x)_{x\geq 0}$ 
of $(X_t)_{t\geq 0}$ is given, for $\lambda \geq 0$, by
\begin{align}
\psi(\lambda)=\frac{-\phi_\lambda'(0)}{\phi_\lambda(0)},
\end{align}
where $\phi_\lambda$ is a decreasing nonnegative solution to
\begin{align}\label{edpy}
\frac12 \phi_\lambda''(x)=\beta(x)\phi'_\lambda(x)+\lambda\phi_\lambda(x).
\end{align}
(With the notation of \cite[Theorem 1]{py}, it holds that $\psi^{0,+}=\psi^{0,-}$
because $\beta$ is odd:  
for $\phi_\lambda(x)$ a decreasing nonnegative solution to
(\ref{edpy}), $\phi_\lambda(-x)$ is an increasing  nonnegative solution to
(\ref{edpy})). 

\vip

Define $a:=\alpha^{-2}\Ai(2\lambda\alpha^{-2}+a_1')/\Ai(a_1')$ 
and $b:=\alpha^{-1}\Ai'(2\lambda\alpha^{-2}+a_1')/\Ai(a_1')$.
We claim that the solution $\phi_\lambda$ to (\ref{edpy}) satisfying 
$\phi_\lambda(0)=a$ and $\phi_\lambda'(0)=b$ 
is nonnegative and decreasing. This will prove that $\psi(\lambda)
=\beta(2\lambda/\alpha^3)$ and end the proof.

\vip

First, this solution is obviously decreasing 
(and thus nonnegative) on $\rr_-$, 
because $\phi_\lambda(0)=a>0$ and $\phi_\lambda'(0)=b<0$,
whence $\phi'_\lambda(x)<0$ for all $x<0$. Indeed, assume by contradiction
that $x_0=\sup\{x<0 : \phi'_\lambda(x)=0\}$ exists. Then $\phi'_\lambda<0$ 
on $(x_0,0)$,
so that $\phi_\lambda>\phi_\lambda(0)>0$ on $(x_0,0]$, whence 
$\phi''_\lambda=2 \beta \phi'_\lambda + \lambda \phi_\lambda>0$
on $(x_0,0)$ (because $\beta<0$ on $\rr_-$), so that 
$\phi'_\lambda(x_0)<\phi'_\lambda(0)<0$,
whence a contradiction. 

\vip

We next observe that for $x\geq 0$, $\phi_\lambda(x)
=\varphi_\lambda(x)$, where $\varphi_\lambda(x):= \alpha^{-2}\Ai
(\alpha x + 2\lambda\alpha^{-2}+a_1')/\Ai(\alpha x + a_1')$.
Indeed, it obviously holds that $\varphi_\lambda(0)=a$ and 
$\varphi_\lambda'(0)=b$ (recall that $\Ai'(a_1')=0$) and some
tedious computations, using only that $\Ai''(y)=y\Ai(y)$, show that 
$\varphi_\lambda$ solves (\ref{edpy}) on $\rr_+$: check successively
that 

\vip

(i) $\varphi_\lambda'(x)=\alpha^{-1}\Ai'
(\alpha x + 2\lambda\alpha^{-2}+a_1')/\Ai(\alpha x + a_1') + \varphi_\lambda(x)
\beta(x)$, 

\vip

(ii) $\varphi_\lambda''(x)=2\lambda \varphi_\lambda(x)
+ 2 \varphi_\lambda'(x) \beta(x)+ \varphi_\lambda(x)[\beta'(x)+
\alpha^2(\alpha x + a_1')-\beta^2(x)]$, 

\vip

(iii) $\beta'(x)=-\alpha^2(\alpha x + a_1')+\beta^2(x)$.

\vip

It only remains to check that $\varphi_\lambda$
is decreasing on $\rr_+$ (it is clearly nonnegative). But we see from (i) above
that 
$\varphi'_\lambda(x)=\varphi_\lambda(x)[\beta(x)-\beta(x+2\lambda \alpha^{-3})]<0$,
since the map $x \mapsto \beta(x)$ is increasing on $\rr_+$ (use e.g.
the second equality in (\ref{edla})).
\end{proof}

\def\refname{References}

\end{document}